\documentclass{article}

\usepackage[utf8]{inputenc} 
\usepackage{url}

\usepackage[dvipsnames,dvipdfmx]{xcolor} 

\usepackage[colorinlistoftodos,prependcaption,textsize=footnotesize]{todonotes}
\usepackage{amsmath,amsfonts,amssymb}
\usepackage{amsthm}
\usepackage{graphicx}
\usepackage{enumerate}
\usepackage{multirow,bigdelim}
\usepackage{mathtools}
\usepackage{soul} 
\usepackage{pstricks}

\usepackage[margin=3cm]{geometry}
\usepackage{array}
\usepackage[normalem]{ulem}
\usepackage{url}
\usepackage{hyperref}
\usepackage{enumerate}
\usepackage[shortlabels]{enumitem}
\usepackage{mathtools}
\hypersetup{
	colorlinks,
	linkcolor={red!50!black},
	citecolor={blue!50!black},
	urlcolor={blue!80!black}
}

\newcommand{\face}{\mathrel{\unlhd}}
\newcommand{\jProd}[2]{ {#1 \circ #2 } }	
\newcommand{\SOC}[2]{{\mathcal{L}^{#2} _{#1}}}

\newcommand{\reInt}{\mathrm{ri}\,}

\newcommand{\spanVec}{\mathrm{span}\,}

\newcommand{\inProd}[2]{\langle #1 , #2 \rangle }

\newcommand{\PSDcone}[1]{{\mathcal{S}^{#1}_+}}

\newcommand{\stdMap}{ {\mathcal{A}}}

\newcommand{\stdCone}{ {\mathcal{K}}}
\newcommand{\stdSpace}{ \mathcal{L}}

\newcommand{\stdFace}{F}

\newcommand{\stdInt}{ {e}}

\newcommand{\matRange}{{\mathrm{ range } \,}}

\renewcommand{\Re}{\mathbb{R}}    
  
\renewcommand{\S}{\mathcal{S}}

\newcommand{\jAlg}{\mathcal{E}}
\newcommand{\T}{*}

\DeclarePairedDelimiter\abs{\lvert}{\rvert}%

\newcommand{\rank}{\operatorname{rank}}
\newcommand{\mult}{\operatorname{mult}}
\newcommand{\Aut}{\operatorname{Aut}}

\newcommand{\Diag}{\operatorname{Diag}}
\newcommand{\GL}{\operatorname{GL}}

\newtheorem{definition}{Definition}[section]
\newtheorem{lemma}[definition]{Lemma}
\newtheorem{proposition}[definition]{Proposition}

\newtheorem{example}[definition]{Example}

\newtheorem{corollary}[definition]{Corollary}

\newtheorem{theorem}[definition]{Theorem}
\newtheorem*{proposition*}{Proposition}
\theoremstyle{remark}
\newtheorem{remark}[definition]{Remark}

\numberwithin{equation}{section}
\title{Automorphisms of rank-one generated hyperbolicity cones and their derivative relaxations}

\author{
	Masaru Ito		\thanks{Department of Mathematics, College of Science and Technology, Nihon University,
		1-8-14 Kanda-Surugadai, Chiyoda-Ku, Tokyo 101-8308, Japan. This author was supported partly by the JSPS Grant-in-Aid for Early-Career Scientists  21K17711.
		Email: \href{ito.masaru@nihon-u.ac.jp}{ito.masaru@nihon-u.ac.jp}.} \and
	Bruno F. Louren\c{c}o\thanks{Department of Statistical Inference and Mathematics, Institute of Statistical Mathematics, Japan.
		This author was supported partly by the JSPS Grant-in-Aid for Early-Career Scientists  19K20217 and the Grant-in-Aid for Scientific Research (B)21H03398.
		Email: \href{bruno@ism.ac.jp}{bruno@ism.ac.jp}.}
}

\begin{document}
\maketitle
\begin{abstract}
	A hyperbolicity cone is said to be rank-one generated (ROG) if all its extreme rays have rank one, where the rank is computed with respect to the underlying hyperbolic polynomial. 
	This is a natural class of hyperbolicity cones which are strictly more general than the ROG spectrahedral cones. 
	In this work, we present a study of the automorphisms of  ROG hyperbolicity cones and their derivative relaxations. One of our main results states that the automorphisms of the derivative relaxations  are exactly the automorphisms of the original cone fixing a certain direction. 
As an application, we completely determine the automorphisms of the derivative relaxations of the nonnegative orthant and of the cone of positive semidefinite matrices. More generally, we also prove relations between the automorphisms of a spectral cone and the underlying permutation-invariant set, which might be of independent interest.
\end{abstract}

\section{Introduction}\label{sec:int}	
In this work, our goal is to prove several results on the automorphism group of  a particular class of hyperbolicity cones  and their derivative relaxations. We start with some general observations about the importance of automorphism groups of cones. An initial motivation is that it provides deep insight into the properties of the cone. As a concrete example, we can consider the notion of \emph{Lyapunov rank} (also known as the \emph{bilinearity rank}) \cite{RNPA11,GT14,GT14_2,OG16,Or22} of a pointed full-dimensional closed convex cone $\stdCone \subseteq \Re^n$, which is the dimension of the Lie algebra of the automorphism group of $\stdCone$ and is denoted by $\beta(\stdCone)$. If $\beta(\stdCone) \geq n$ holds, then $\stdCone$ is said to be a perfect cone \cite[Theorem~1]{OG16} and a complementarity condition such as 
``$0 = \inProd{x}{y}$, $x \in \stdCone, y \in \stdCone^*$'' (here, $\stdCone^*$ is the dual cone of $\stdCone$) can be rewritten as a square system of equations \cite[Proposition~2]{GT14_2}\footnote{One example of this phenomenon is how the conditions $\inProd{x}{y} = 0, x \in \Re^n_+, y \in \Re^n_+$ imply $n$ equations $x_i y_i = 0$. }. This influences how easy (or how hard) it is to solve complementarity problems involving $\stdCone$.   
More generally, 
it can be shown that $L$ belongs to the Lie algebra of the automorphism group of $\stdCone$ if and only if the following condition holds, see \cite[pg.~157]{GT14_2}.
\[
x \in \stdCone, y \in \stdCone^*, \inProd{x}{y} = 0 \Rightarrow \inProd{L(x)}{y} = 0.
\]

Another interesting application of the study of automorphism groups is a classification of three-dimensional cones that have an automorphism group of dimension at least two \cite{Hi14}. A key step is the analysis of the action of a certain parametric subgroup on the elements of the cone, see \cite[pg.~501]{Hi14}.

There are also more practical concerns related to the efficient optimization over the underlying cone. When the problem data has some amount of symmetry, it might be possible to use \emph{symmetry-reduction} techniques in order to decrease the size of the problem or prove some favourable property, see   \cite{KOMK01,KP04,KS08,Va09} for examples in semidefinite programming and polynomial optimization. 
This is contingent, of course, on having a good grasp of the linear transformations that preserve the cone. 
The nonnegative orthant $\Re^n_+$ and the positive semidefinite cone $\PSDcone{n}$ are rich in automorphisms, this provides fertile ground for techniques that exploit the symmetries in problem data.

We now return to our subject matter: hyperbolic polynomials and hyperbolicity cones. We recall that a cone $\stdCone$ is said to be \emph{homogeneous} if its group of automorphisms acts transitively on the relative interior of $\stdCone$. Furthermore, a cone is said to be \emph{symmetric} if it is homogeneous and self-dual with respect to some inner product.
Typical examples of symmetric cones include the aforementioned $\Re^n_+$, $\PSDcone{n}$ but also the second-order cones, also known as Lorentz cones.
To the best of our knowledge, G\"uler \cite{Gu97} was the first to bring hyperbolic polynomials to the attention of optimizers in the form of hyperbolic programming. 
One of his motivations was to try to extend certain types of long-step interior point methods from symmetric cones to other classes of cones. Among many results, G\"uler proved that all homogeneous cones are hyperbolicity cones \cite[Section~8]{Gu97}.

From the conic optimization point of view, hyperbolicity cones are the next natural step after symmetric cones and homogeneous cones. And, of course, they have been subject of much recent research activity. There are deep questions in convex algebraic geometry related to hyperbolicity cones, such as the generalized Lax conjecture.
Related to that, for many years, one of the best results on the generalized Lax conjecture was the one proved by Chua (see the comment in \cite[p.~64]{Re06}): all homogeneous cones are spectrahedral  \cite{CH03}, see also \cite[Proposition 1 and Section 4]{FB02} by Faybusovich.

This kind of striking result is only possible thanks to powerful algebraic theories for homogeneous cones (T-Algebras \cite{V63}) and symmetric cones (Euclidean Jordan Algebras \cite{FK94}), which are viable because these cones have ``large'' automorphism groups. Indeed, the Lyapunov rank of a symmetric cone must be at least the dimension of the underlying space, see \cite[Theorem~5]{GT14_2}.

From this point of view, it seems natural to try to understand the automorphism group of hyperbolicity cones. We would like to understand how large can it be, what can it tell us about the structure of the underlying cone and so on.

Unfortunately, a significant hurdle in this enterprise is that  determining the automorphisms of any given mathematical  object is typically very hard. 
In order to appreciate the difficulty in this task, we may recall that every polyhedral cone can be realized as a hyperbolicity cone. 
Therefore, a general theory of automorphisms of hyperbolicity cones would need to contemplate the automorphisms of all polyhedral cones, so this is a non-starter.

Fortunately, we found a class of hyperbolicity cones that are quite suitable for the study of its automorphisms: the rank-one generated (ROG) hyperbolicity cones together with their derivative relaxations.

A (pointed) hyperbolicity cone is said to be ROG if all its extreme rays have rank one, when computed with respect to the underlying hyperbolic polynomial. We recall that a spectrahedral cone is said to be ROG if all its extreme rays are generated by matrices of rank $1$, e.g., see \cite{Hd16}. In particular, all  ROG spectrahedral cones are also ROG hyperbolicity cones if the hyperbolic polynomial is appropriately chosen, see Section~\ref{sec:rog_ex}. The 
ROG spectrahedral cones themselves are quite important because they are connected to whether SDP relaxations of certain quadratic programs are exact or not, see \cite{Hd16,AKW22}.
However, we note that ROG hyperbolicity  cones form a strictly larger class of cones, because the second-order cones in dimension $4$ or larger cannot be realized as ROG spectrahedral cones,  see Proposition~\ref{prop:socp_not_rog}.

Our main results are as follows.
\begin{itemize}
	\item We provide a detailed study of the properties and facial structure of ROG hyperbolicity cones. Given a regular ROG hyperbolicity cone $\Lambda_+$ of dimension at least three and generated by a polynomial $p$ along a direction $e$, we show in Theorem~\ref{theo:hyper_aut_2},  the  formula 
	\[
	\Aut(\Lambda_+^{(k)}) = \{A \in \Aut(\Lambda_+) \mid A(\Re_+e)=\Re_+e \},
	\]
	for $k$ satisfying $1 \leq k \leq \deg p -3$. That is, the automorphisms of the $k$-th derivative relaxation $\Lambda_+^{(k)}$ are precisely the automorphisms of $\Lambda_+$ that have the hyperbolic direction $e$ as an eigenvector.
	The proof requires both geometric and algebraic considerations and makes full use of G{\aa}rding's inequality.
	Surprisingly, Theorem~\ref{theo:hyper_aut_2} also admits a converse of sorts and, with some caveats, it is possible to show that 
	the automorphisms of $\Lambda_+$ must also be automorphisms of derivative relaxations of faces of $\Lambda_+$, see Theorem~\ref{theo:aut_converse}.
	
	\item With the aid of Theorem~\ref{theo:hyper_aut_2} we completely determine the automorphism groups of the derivative relaxations of $\PSDcone{n}$ and $\Re^n_+$ in Theorems~\ref{th:aut-orth} and \ref{theo:aut_snk}, respectively. The derivative relaxations of  $\PSDcone{n}$  and $\Re^n_+$ are related in the sense that the former is the \emph{spectral cone} generated by the latter. Related to that, we also prove a general result on the automorphisms of spectral cones, see Theorem~\ref{thm:spec_aut}.
	\item Finally, in Section~\ref{sec:non_hom}, we show some corollaries of our results. We compute the Lyapunov ranks of derivative relaxations of $\PSDcone{n}$  and $\Re^n_+$. We also discuss the non-homogeneity of derivative relaxations of ROG cones. 
\end{itemize}

This work is organized as follows. 
In Section~\ref{sec:prel}, we discuss the notation and some basic tools from convex analysis and hyperbolicity cones. In Section~\ref{sec:aut}, we discuss ROG hyperbolicity cones  and prove our main result on their automorphisms.
In Section~\ref{sec:app}, we discuss applications of our results. Finally, in Section~\ref{sec:conc} we conclude this work with some open questions.

\section{Preliminaries}\label{sec:prel}

Let $C \subseteq \Re^n$ be a convex set. We denote the relative interior, span and dimension of $C$  by $\reInt C$, $\spanVec C$,  and $\dim C$ respectively.

Next, let $\stdCone \subseteq \Re^n$ be a closed convex cone. We say that $\stdCone$ is \emph{pointed} if 
$\stdCone \cap - \stdCone = \{0\}$. If 
$\dim \stdCone = n$, then $\stdCone$ is said to be full-dimensional. If $\stdCone$ is pointed and full-dimensional we say that it is \emph{regular}.
A \emph{face} of $\stdCone \subseteq \Re^n$ is a convex cone $\stdFace \subseteq \stdCone$ such that every $x,y \in \stdCone$ with $x+y \in \stdFace$ satisfies $x,y \in \stdFace$. In this case, we write $\stdFace \face \stdCone$. 
If $\stdFace \neq \stdCone$, then we say that $\stdFace$ is a \emph{boundary face}.
A face $\stdFace$ is said to be \emph{exposed}, if it can be written as the intersection of $\stdCone$ with one of its supporting hyperplanes. An \emph{extreme ray} is a face of $\stdCone$ of dimension $1$. If $\stdFace = \{\alpha x \mid \alpha \geq 0\} \face \stdCone$ we say that $\stdFace$ is the extreme ray \emph{generated by $x$}.

We now recall some basic properties of faces, more details can be seen in \cite{Pa00}. 
A useful property is that for two faces $\stdFace_1 \face \stdCone$, $\stdFace_{2} \face \stdCone$, we have 
\begin{equation}\label{eq:faces_eq}
\stdFace_{1} = \stdFace_{2} \quad \Longleftrightarrow \quad (\reInt \stdFace _1) \cap (\reInt \stdFace_{2}) \neq \emptyset,
\end{equation}
see \cite[Corollary~18.1.2]{RT97}. Also, for every convex subset $S \subseteq \stdCone$, there exists a unique face $\stdFace \face \stdCone$ such that 
$S \subseteq \stdFace$ and $(\reInt S)\cap (\reInt \stdFace) \neq \emptyset$. This is called the \emph{minimal face of $\stdCone$ containing $S$} and will be denoted by $\stdFace(S)$. In particular, $\stdFace(S)$ is the intersection of all faces of $\stdCone$ containing $S$ and for a face $\hat \stdFace \face \stdCone$ we have
\begin{equation}\label{eq:fs}
\hat \stdFace = \stdFace(S)  \quad \Longleftrightarrow \quad (\reInt S) \cap (\reInt \hat \stdFace) \neq \emptyset \quad \Longleftrightarrow\quad \reInt S \subseteq \reInt \hat \stdFace,
\end{equation}
see \cite[Proposition~3.2.2]{Pa00} or it can also be inferred from the  results in \cite[Section~18]{RT97}.

If $S$ is the convex hull of finitely many points $\{x_1,\ldots, x_{r}\} \subseteq \stdCone$ we will simplify the notation and write 
$\stdFace(x_1,\ldots, x_r)$. The following well-known lemma will be useful and is
closely  related to  \cite[Proposition~3.3]{Bar73}.
\begin{lemma}\label{lem:min_face}
Let $\{x_1,\ldots, x_r\} \subseteq \stdCone$ be a finite subset of points of a closed convex cone $\stdCone$. Then 
\[
\stdFace(x_1,\ldots, x_r) = \stdFace(x_1+\cdots+x_r)
\]
\end{lemma} 
\begin{proof}
By definition, $x_1+\cdots + x_r \in \reInt \stdFace(x_1+\cdots+x_r)$. Since $\reInt \stdFace(x_1+\cdots+x_r)$ is a cone we also have 
$(x_1+\cdots + x_r)/r \in \reInt \stdFace(x_1+\cdots+x_r)$.
Next, let $S$ be the convex hull of $x_1,\ldots, x_r$.
Then, \eqref{eq:fs} implies that $\reInt S \subseteq \reInt\stdFace(x_1,\ldots, x_r)$.

However, since  $(x_1+\cdots + x_r)/r \in \reInt S$ holds\footnote{A quick way to see that is to note that $S$ is the image of the unit simplex $P \coloneqq \{(\alpha_1,\ldots, \alpha_r) \in \Re^{r}\mid \alpha_i \geq 0, \alpha_1+\cdots \alpha_r = 1\}$ by the linear map $A$ that takes the $i$-th unit vector in $\Re^r$ to $x_i$. Then, recalling that $\reInt AP = A (\reInt P)$ (\cite[Theorem~6.6]{RT97}) leads to the desired conclusion.}, we conclude that $\reInt \stdFace(x_1,\ldots, x_r)$ and $\reInt \stdFace(x_1+\cdots+ x_r)$ intersect, so they are equal by \eqref{eq:faces_eq}.
\end{proof}

The \emph{automorphism group} of a cone $\stdCone$ is the set
$$
\Aut(\stdCone) := \{A \in \GL_n(\Re) \mid A\stdCone=\stdCone\},
$$
which is a subgroup of $\GL_n(\Re) \coloneqq \{A \in \Re^{n\times n} \mid \det A \ne 0\}$.

The $(n+1)$-dimensional second-order cone in $\Re^{n+1}$ is denoted by $\SOC{2}{n+1} \coloneqq \{(x_0,\bar{x}) \in \Re\times \Re^{n} \mid x_0 \geq 0, x_0^2 \geq x_1^2 + \cdots + x_n^2 \}$.
The space of $n\times n$ real symmetric matrices is denoted by $\S^n$ and the cone of $n\times n$ real symmetric positive semidefinite matrices is denoted by $\PSDcone{n}$.
If $x \in \Re^n$, we denote by $\Diag(x)$ the diagonal matrix in $\S^n$ having $x$ as its diagonal. 


\subsection{Hyperbolicity cones}
We review some basic 
facts about hyperbolicity cones. More details can be seen in \cite{Gu97,BGLS01,Re06}. Let $p:\Re^n \to \Re$ be an homogeneous polynomial of degree $d \coloneqq \deg p$. If $e \in \Re^n$ is such that $p(e) > 0$ and for every $x$ 
the polynomial $t \mapsto p(x -te)$ only has real roots, then $p$ is said to be \emph{hyperbolic along $e$}. The corresponding hyperbolicity cone is given by
\begin{equation}
\label{eq:hypcone}
\Lambda_+(p,e) \coloneqq \{x\in \Re^n \mid \textup{all roots of $t\mapsto p(te-x)$ are nonnegative}\}.
\end{equation}
We will sometimes omit $(p,e)$ and simply write $\Lambda_+$.

Given $x \in \Re^n$, the roots of $t\mapsto p(te-x)$ are called \emph{eigenvalues} of $x$ and denoted by
$
\lambda_1(x) \geq \lambda_2(x) \geq \cdots \geq \lambda_d(x),
$
counting with multiplicity. We also define $\lambda:\Re^n\to\Re^d$ by
$$
\lambda(x) \coloneqq (\lambda_1(x),\lambda_2(x),\ldots,\lambda_d(x))^T.
$$

Let $D_e^kp(x)= \frac{d^k}{dt^k}p(x+te)|_{t=0}$ denote the $k$-th directional derivative of $p$ in the direction $e$, which becomes a hyperbolic polynomial of degree $d-k$ along $e$ \cite[Lemma 1]{Gar}.
With that, we define 
the \emph{$k$-th derivative relaxation} (also called \emph{Renegar} derivative)  of 
$\Lambda_+(p,e) $ by
\begin{equation}\label{eq:hyp_deriv_def}
\Lambda_+^{(k)} \coloneqq \Lambda_+(D_e^kp,e).
\end{equation}
We note that $\Lambda_+^{(k)}$
can be alternatively written as (cf. \cite[Proposition 18, Corollary 19 and Theorem 20]{Re06})
\begin{equation}\label{eq:hyp_deriv}
\Lambda_+^{(k)} \coloneqq \{x\in \Re^n \mid D_e^k p(x) \geq 0,\; D_e^{k+1}p(x) \geq 0, \;\ldots,\; D_e^{d-1}p(x) \geq 0\}.
\end{equation}
Furthermore \eqref{eq:hyp_deriv} implies that
$\Lambda_+^{(0)} \subseteq \cdots \subseteq  \Lambda_+^{(d-1)}$ holds.
Next, we recall some facts on the faces of hyperbolicity cones.

\begin{lemma}[{Renegar \cite[Proposition 24]{Re06}}]\label{lem:face}
Suppose that $\Lambda_+ = \Lambda_+(p,e)$ is regular.
For $i=0,1,\ldots,\deg p -2$, every boundary face of $\Lambda_+^{(i)}$ is either a face of $\Lambda_+$ or an (exposed) extreme ray not contained in $\Lambda_+$.
\end{lemma}

\begin{lemma}[{Renegar \cite[Proposition 25]{Re06}}]\label{lem:face-mul}
Let $F$ be a boundary face of $\Lambda_+$ other than its lineality space $\Lambda_+ \cap (-\Lambda_+)$.
For any $x \in \reInt \stdFace$, let $m$ be the multiplicity of zero as an eigenvalue. Then, $F$ is a face of $\Lambda_+^{(m-1)}$.
\end{lemma}

\paragraph{Rank function over hyperbolicity cones}
For $x \in \Lambda_+(p,e)$, we define the $\rank(x)$ as 
$d - m$, where $m$ is the multiplicity of $0$ as an eigenvalue of $x$.
In other words, $\rank(x)$ is the number of nonzero eigenvalues of $x$.
For convenience we also define $\mult(x) \coloneqq m$ so that
\[
\rank(x) + \mult(x) = d.\]
We recall that $p$ is also hyperbolic along every $\hat e\in \reInt \Lambda_+(p,e)$
and that $\Lambda_+(p,e) = \Lambda_+(p,\hat e)$, see \cite[Theorem~3]{Re06}.
Although the eigenvalue function $\lambda$ depends on the chosen hyperbolicity direction, the $\rank$ and $\mult$ functions do not, see \cite[Proposition~22]{Re06}.
So $\rank$ and $\mult$ only depend on $p$ and the underlying cone, but not on the choice of hyperbolicity direction (as long as the direction belongs to the relative interior of the cone).

If $F$ is a face of $\Lambda_+(p,e)$, we define the rank of $F$ as the maximum of $\rank(x)$ over $F$. The rank function has the following well-known properties.

\begin{proposition}[Properties of rank]\label{prop:rank}
	Let $x,y \in \Lambda_+$ and $F,F'$ be faces of $\Lambda_+$.
	\begin{enumerate}[$(i)$]
		\item\label{prop:rank:1} $\rank(x+y) \leq \rank(x) + \rank(y)$.
		\item\label{prop:rank:2} Suppose $x \in \stdFace$, then $x \in \reInt F$ if and only if 
		$\rank(x) = \rank(\stdFace)$
		\item\label{prop:rank:3} $\rank(F) < \rank(F')$ holds when $F \subsetneq F'$.
		\item \label{prop:rank:4} If $\Lambda_+$ is pointed and $x \in \Lambda_+$, then $\rank(x) = 0 \Leftrightarrow x = 0$.
	\end{enumerate}		
\end{proposition}
\begin{proof}
	Item~\ref{prop:rank:1} is true if $x = y$, because $\rank(2x) = \rank(x)$.
	Next, we assume $x \neq y$ and let $\mathcal{E} \coloneqq \{x,y\}$. It is shown in \cite[Proposition~3.2]{Br11}
	that the function $r: 2^\mathcal{E} \to \mathbb{N}$ defined by 
	$r(\emptyset) \coloneqq 0$ and
	\[
	r(S) \coloneqq \rank\left(\sum _{z \in S} z \right)
	\]
	is a polymatroid and, therefore, submodular.
	We have
	\[r(\{x\}\cup \{y\} ) + r(\{x\}\cap \{y\} ) \leq  r(\{x\} )+r(\{y\}).
	\] 
	That is,   $\rank(x+y) \leq \rank(x) + \rank(y)$.
	This proves item~\ref{prop:rank:1}.
	
	Item~\ref{prop:rank:2} is direct consequence of \cite[Theorem~26]{Re06} which gives an analogous statement for the $\mult(\cdot)$ function.
	Moreover, item~\ref{prop:rank:2} combined with \eqref{eq:faces_eq} leads to item~\ref{prop:rank:3}.
	Finally, item~\ref{prop:rank:4} follows from \cite[Proposition~11]{Re06}, which 
	states that $\mult(x) = d$ if and only if $x $ belongs to the lineality space of $\Lambda_+$. Since $\Lambda_+$ is pointed, the lineality space is 
	$\{0\}$. 
\end{proof}

\paragraph{Minimal polynomials}
If $\Lambda_+$ is a hyperbolicity cone, there could be several polynomials of different degrees satisfying $\Lambda_+ = \Lambda_+(p,e)$. However, the polynomial of minimal degree that generates $\Lambda_+$ is unique up to scaling by a positive constant. The precise result is as follows.
\begin{proposition}[Helton and Vinnikov, {\cite[Lemma~2.1]{HV07}}]\label{prop:hv}
	A homogeneous hyperbolic polynomial $p$ of minimal degree such that $\Lambda_+ = \Lambda_+(p,e)$ is unique up to multiplication by a positive constant. If $\Lambda_+(p,e) = \Lambda_+(q,e)$, then $q = ph$, 
	where $h$ is a polynomial that is strictly positive on a dense connected subset of $\Lambda_+(p,e)$.
\end{proposition}

\paragraph{Basic properties of automorphisms of hyperbolicity cones}

We start our explorations on automorphisms with the following basic result, which connects the automorphisms of a hyperbolicity cone with the underlying hyperbolic polynomial, provided that it is minimal.

\begin{proposition}\label{prop:aut}
	Suppose that $\Lambda_+ = \Lambda_+(p,e)$ where $p$ is of minimal degree and let $A \in \GL_n(\Re)$. Then 
	$A \in  \Aut(\Lambda_+)$ if and only if $Ae \in \reInt \Lambda_+$ and there exists a positive constant $\kappa$ such that $p = \kappa (p\circ A)$.	
\end{proposition}
\begin{proof}
	First, suppose that $A \in   \Aut(\Lambda_+)$ and let $q \coloneqq p \circ A$. 
	Because $A$ is an automorphism and 
	$e \in \reInt \Lambda_+$, then $Ae \in \reInt\Lambda_{+}$, so 	$p$ is also hyperbolic along $Ae$, e.g., see \cite[Theorem~3]{Re06}.
	This tells us that $q = p \circ A$ is hyperbolic along $e$.
	
	The roots of $t \mapsto q(x-te)$ are nonnegative if and only if the roots of $t \mapsto p(Ax-tAe)$ are nonnegative. That is, if and only if $Ax \in \Lambda_+$. But, since $
	A$ is an automorphism, this happens if and only if $x \in \Lambda_+$.
	We conclude that $\Lambda_+ = \Lambda_+(q,e)$.	
	The polynomial $p$ is minimal and $q$ has the same degree as $p$, so Proposition~\ref{prop:hv} tells us that there exists a positive constant $\kappa$ such that 
	\[
	p = \kappa (p\circ A).
	\] 
	
Next, we prove the converse, so suppose that $A \in \GL_n(\Re)$ is such that $Ae \in \reInt \Lambda_+$ and there exists a positive constant $\kappa$ such that $p = \kappa (p\circ A)$. Then, 
$p$ and $p\circ A$ are both hyperbolic along $e$ and generate the same hyperbolicity cone $\Lambda_+(p,e) = \Lambda_+(p\circ A,e)$.

Let $x\in \Re^n$ be arbitrary.
Then, $\Lambda_+(p,e) = \Lambda_+(p\circ A,e)$ implies that the roots of 
$t \mapsto p(x - te)$ are nonnegative if and only if the roots of $t \mapsto p(Ax - tAe)$ are nonnegative.
Next, we recall that the hyperbolicity cone stays the same if another relative interior direction is chosen as hyperbolic direction (\cite[Theorem~3]{Re06}).
Since 
$Ae \in \reInt \Lambda_+$,  
the roots of $t \mapsto p(Ax - tAe)$ are nonnegative if and only if 
the roots of $t \mapsto p(Ax - te)$ are nonnegative.
The overall conclusion is that $x \in \Lambda_+$ if and only if $Ax \in \Lambda_+$. 
That is, 
$A \in \Aut(\Lambda_+)$.
\end{proof}

Next, we would like to understand how $\rank(Ax)$ and $\rank(x)$ are related if $x \in \Lambda_+$ and $A \in \Aut(\Lambda_+)$. 
Unfortunately, in general, 
$\rank(x) \neq \rank(Ax)$.
\begin{example}[Automorphisms do not necessarily preserve the hyperbolic rank]\label{ex:rank}
	We have $\Re_+^3 = \Lambda_+(\tilde p,(1,1,1))$, where $\tilde p(x_1,x_2,x_3) \coloneqq x_1^2 x_2 x_3$.
	Let $x = (1,0,0)$ and $y = (0,1,0)$. Then the eigenvalue vectors of 
	$x$ and $y$ with respect to $\tilde p$ are $\lambda(x) = (1,1,0,0)$ and 
	$\lambda(y) = (0,0,1,0)$. 
	Therefore, $\rank(x) = 2$ and $\rank(y) = 1$. With that, the linear map $A:\Re^3 \to \Re^3$ that exchanges $x_1$ and $x_2$ and fixes $x_3$ is an automorphism of $\Re_+^3$ such that $Ax = y$ and $\rank(x) \neq \rank(Ax)$.
\end{example}
In Example~\ref{ex:rank} the rank is not preserved under automorphisms because of redundancies in how $\Re_+^3$ is realized as a hyperbolicity cone. If we had used $p(x_1,x_2,x_3)  \coloneqq x_1x_2x_3$ there would be no such problem.

\begin{proposition}\label{prop:aut_rank}
	Suppose that $\Lambda_+ = \Lambda_+(p,e)$ where $p$ is of minimal degree and let $A \in \Aut(\Lambda_+)$. If $x \in \Lambda_+$ and $A \in \Aut(\Lambda_+)$, then $\rank(Ax) = \rank(x)$.	
\end{proposition}
\begin{proof}
Proposition~\ref{prop:aut} tells us that there exists a positive constant $\kappa$ such that 
\[
p = \kappa (p\circ A).
\] 
This implies that $t \mapsto p(Ax -tAe)$ and  $t \mapsto p(x -te)$ have the same number of positive roots.
Furthermore, the rank does not change if another relative interior point is chosen as direction of hyperbolicity (\cite[Proposition~22]{Re06}), 
so the number of positive roots of 
$t \mapsto p(Ax -tAe)$ and 
$t \mapsto p(Ax -te)$ coincide. 
Therefore, $\rank(x) = \rank(Ax)$.
\end{proof}

\section{ROG hyperbolicity cones, derivative relaxations and automorphisms}\label{sec:aut}

In this section, we present a study of ROG hyperbolicity cones and their derivative relaxations. 
We start with some basic properties and examples in order to set the stage for a study of their automorphisms.

\subsection{ROG hyperbolicity cones and their facial structure}

\begin{definition}[Rank-one generated hyperbolicity cones]
	A pointed hyperbolicity cone $\Lambda_+$ is said to be \emph{rank-one generated (ROG) with respect to $p$ and $e$} if $\Lambda_+ = \Lambda_+(p,e)$ and all extreme rays of $\Lambda_+$ are generated by rank $1$  elements (computed with respect to $p$ and $e$). 
\end{definition}
When it is clear from the context we will simply say that \emph{$\Lambda_+$ is ROG} and omit the reference to $p$ and $e$.
\begin{remark}
	Being ROG is a property that depends on the choice of $p$. For example, let $p$ and $\tilde p$ be such that $p(x_1,x_2,x_3) \coloneqq x_1 x_2 x_3$ and 
	$ \tilde p(x_1,x_2,x_3) \coloneqq x_1^2 x_2 x_3$.
	Then, $\Re_+^3 = \Lambda_+(p,(1,1,1))=\Lambda_+(\tilde p,(1,1,1))$ and $\Re_+^3$ is ROG with 
	respect to $p$ but not with respect to $\tilde p$. 
\end{remark}

We say that a closed convex cone $\stdCone$ is \emph{strictly convex}
if its only faces are $\{0\}, \stdCone$ and extreme rays. 
We have the following lemma.

\begin{lemma}[Extreme rays and strictly convex faces]\label{lem:rank}
	Suppose that $\Lambda_+=\Lambda_+(p,e)$ is regular and ROG with respect to $p$.
	Then the following assertions hold.
	\begin{enumerate}[$(i)$]
		\item\label{lem:rank:1} If $\stdFace \face  \Lambda_+$ is a face of rank one, then $\stdFace$ is an extreme ray.
		\item \label{lem:rank:2} If $x,y \in \Lambda_+$ are linearly independent and have rank $1$ then $\rank(x+y) = 2$.
		\item\label{lem:rank:3} Let $\stdFace \face \Lambda_+$ be a face such that $\dim \stdFace \geq 2$.
		Then, $\stdFace$ is strictly convex if and only if 
		$\rank (\stdFace) = 2$.
	\end{enumerate}
	
\end{lemma}
\begin{proof}
	\begin{enumerate}[$(i)$]
		\item 
		Suppose that $\stdFace$ is not an extreme ray. 
		Then, $\stdFace$ must contain at least one  extreme ray $\hat \stdFace$ such that $\hat \stdFace$ is properly contained in $\stdFace$.  $\hat \stdFace$ is a face as well, so
		by item~\ref{prop:rank:3} of Proposition~\ref{prop:rank}, we must have 	$\rank(\hat \stdFace) = 0$.
		Since $\Lambda_+$ is pointed this implies that $\hat \stdFace = \{0\}$ by item~\ref{prop:rank:4} of Proposition~\ref{prop:rank}, which  contradicts the fact that $\hat \stdFace$ should have dimension $1$.
		\item By item~\ref{prop:rank:1} of Proposition~\ref{prop:rank}, $\rank(x+y) \leq 2$. 
		We have 
		$x+y \in \reInt \stdFace(x+y)$, where we recall that $\stdFace(x+y)$
		is  the minimal face of $\Lambda_+$ containing $x+y$. 
		With that, we have $\rank(\stdFace(x+y)) \leq 2$ by item~\ref{prop:rank:2} of Proposition~\ref{prop:rank}.
		By Lemma~\ref{lem:min_face}, we have $\stdFace(x+y) = \stdFace(x,y)$, so $\stdFace(x+y)$ has at least two distinct extreme rays, since $x,y$ are linearly independent.
		In view of  \ref{lem:rank:1}, the rank of $\stdFace(x+y)$
	cannot be $1$.

		So, $\rank \stdFace(x,y) = 2$ and since $x+y \in \reInt \stdFace(x+y)$, 
		Proposition~\ref{prop:rank} tells us that indeed $\rank(x+y) = 2$.
		
		\item If $\stdFace$ is strictly convex, then the sum of two linearly independent extreme rays is in 
		$\reInt \stdFace$ and it must have rank $2$ by item~\ref{lem:rank:2}.
		Conversely, if $\stdFace \face \Lambda_+$ has rank $2$, by item~\ref{prop:rank:3} of Proposition~\ref{prop:rank},  the boundary faces of $\stdFace$ must be of rank $1$ or $0$, which are extreme rays (by item~\ref{lem:rank:1}) or the face $\{0\}$ (by item~\ref{prop:rank:4} of Proposition~\ref{prop:rank}), respectively.
	\end{enumerate}
\end{proof}

Our next task is showing that if $\Lambda_+$ is ROG with respect to $p$ and $e$, then $p$ must already be the minimal degree polynomial for $\Lambda_+$ and this property gets propagated to the derivative relaxations.
First, we prove the following result which is a slightly refined version of \cite[Proposition~2.2]{Sa20}.
\begin{proposition}\label{prop:chain2}
	Suppose $\Lambda_+ = \Lambda_+(p,e)$ is  regular and ROG with respect to $p$. Suppose $\stdFace \face \Lambda_+$ is a face of rank $r$ and let $\hat \stdFace$ be any extreme ray of $\stdFace$.
	Then, there exists a chain of faces of $\Lambda_+$ with  length $r+1$  such that 
	$\{0\} = \stdFace_0 \subsetneq \stdFace_1 = \hat \stdFace \subsetneq \stdFace_2 \subsetneq \cdots \subsetneq \stdFace_{r-1} \subsetneq \stdFace_{r} = \stdFace$ {and $\rank(F_i) = i$ ($i=0,1,\ldots,r$)}.
\end{proposition}
\begin{proof}
	 The result is clear for faces of rank $r \leq 1$. So suppose that $\stdFace$ has rank $r \geq 2$ and we will build a chain of faces from the bottom up. 
	Let $x_1 \in \stdFace$ be such that $x_1$ generates the  extreme ray $\hat \stdFace$ and set $\stdFace_0 \coloneqq \{0\}, \stdFace_1 \coloneqq \stdFace(x_1) = \hat \stdFace$.
	It is clear that $\stdFace_0 \subsetneq \stdFace_1$ by $\dim(\stdFace_0)=0\ne 1 = \dim(\stdFace_1)$.
	
	Then, for $i \geq 2$  we proceed as follows. If $\rank(\stdFace_{i-1}) = r-1$, we let $\stdFace_i \coloneqq  \stdFace$ and we stop. Otherwise, since $\stdFace_{i-1} \neq \stdFace$, there must be an extreme ray $x_{i}$ of $\stdFace$ not contained in $\stdFace_{i-1}$, so we let 
	$\stdFace_{i} \coloneqq \stdFace(x_1,\ldots, x_{i})$. Since $x_{i} \not \in \stdFace_{i-1}$, we have 
	\[\stdFace_{i-1} \subsetneq \stdFace_{i}.\] 
	By construction, the rank of $\stdFace_{i-1}$ is ${i-1}$ and we will show that the rank of $\stdFace_{i}$ is $i$. 
	Now, by Lemma~\ref{lem:min_face} we have 
	\[
	\stdFace(x_1+\cdots +x_{i}) = \stdFace(x_1, \ldots, x_{i}),
	\]
	Furthermore, $x_1+\cdots +x_{i} \in \reInt \stdFace(x_1+\cdots + x_i) = \reInt \stdFace_i$, so, by item~\ref{prop:rank:2} of Proposition~\ref{prop:rank}, it suffices to compute the rank of 
	$x_1+\cdots +x_{i}$. Item~\ref{prop:rank:1} of Proposition~\ref{prop:rank} tells us that
	\[
	\rank(x_1+\cdots +x_{i}) \leq \rank(x_1+\cdots +x_{i-1}) + \rank(x_{i}) = i.
	\]
	However, 
	$\rank(x_1+\cdots +x_{i})$ is at least $i-1$, since $\rank(x_1+\cdots +x_{i}) = \rank(\stdFace_{i})$ and $\stdFace_{i}$ contains $\stdFace_{i-1}$.
	
	For the sake of obtaining a contradiction suppose that $\rank(x_1+\cdots +x_{i}) = i-1$ holds. Then, we would have $\rank(\stdFace_{i}) = \rank(\stdFace_{i-1})$, which would imply that 
	$x_1+\cdots + x_{i-1} \in \reInt \stdFace_{i}$, by item~\ref{prop:rank:2} of Proposition~\ref{prop:rank}.
	In particular, $(\reInt \stdFace_{i}) \cap (\reInt \stdFace_{i-1}) \neq \emptyset$  would hold which leads to $\stdFace_{i} = \stdFace_{i-1}$, by \eqref{eq:faces_eq}.
	This contradicts
	$\stdFace_{i-1} \subsetneq \stdFace_{i}$.
	
	Therefore,  $\rank(x_1+\cdots +x_{i}) = i$ and $\stdFace_{i}$ is indeed a face of rank $i$. We conclude that each time a new face is added to the chain, the rank increases by exactly one.  Furthermore, we can always keep adding a new face as long as $\stdFace_{i} \neq \stdFace$. Since $\stdFace$ has rank $r$, this leads to a chain of faces of length exactly $r+1$.
\end{proof}

%

\begin{proposition}[Minimal polynomial of ROG cones and their derivative relaxations]\label{prop:strict}
	Suppose that $\Lambda_+(p,e)$ is regular and ROG with respect to $p$.
	Then the following items hold.
	\begin{enumerate}[$(i)$]
		\item \label{prop:strict:1} $p$ is a minimal degree polynomial for $\Lambda_+(p,e)$.
		\item \label{prop:strict:2} $D_{e}^kp$ is a minimal degree polynomial for $\Lambda_+^{(k)}(p,e)$ for all $1\leq k\leq d-2$
		\item \label{prop:strict:3} For all $1 \leq k \leq d-2$, $\Lambda_+^{(k-1)}$ is strictly contained in $\Lambda_+^{(k)}$.
	\end{enumerate}
\end{proposition}
\begin{proof}
	Let $d$ be the degree of $p$. 
	Then, taking  $\stdFace = \Lambda_+(p,e)$ in
	Proposition~\ref{prop:chain2} and letting $\hat \stdFace$ be an arbitrary extreme ray, there exists a 
	chain of faces of length $d+1$
	\[
	\{0\} = \stdFace_0 \subsetneq \stdFace_1 = \hat\stdFace \subsetneq \cdots \subsetneq \stdFace_d = \Lambda_+(p,e),
	\]
	{such that $\rank(F_i)=i$}.
	Now, suppose that there is some hyperbolic polynomial $\tilde p$ of degree $\tilde d \leq d$ such that $\Lambda_+(p,e) = \Lambda_+(\tilde p,e)$.
	{In view of item~\ref{prop:rank:3} of Proposition~\ref{prop:rank}}, the longest possible chain of faces  that $\Lambda_+(p,e) = \Lambda_+(\tilde p,e)$ can have has length at most $\tilde d +1$, which implies that
	$d+1 \leq \tilde d +1$. This proves $d=\tilde{d}$ and so item~\ref{prop:strict:1} holds.
	
	Next, we fix $k$ satisfying $1 \leq k \leq d-2$ and put $i\coloneqq d-(k+1) \geq 1$.
	Since $\rank(\stdFace_i) = i$ we have $\mult(x)=d-i=k+1$ for any $x \in \reInt \stdFace_i$, so by Lemma~\ref{lem:face-mul},
	$\stdFace_{i}$ is a boundary face of $\Lambda_+^{(d-i)-1}(p,e) = \Lambda_+^{(k)}(p,e)$.
	Then, since $\stdFace_0,\stdFace_1,\ldots,\stdFace_{i-1}$ are faces of $\stdFace_i$, they are faces of $\Lambda_+^{(k)}(p,e)$ as well.
 We have thus proved that 
	$\Lambda_+^{(k)}(p,e)$ contains the following chain of faces
	\[
	\stdFace_0 \subsetneq \stdFace_1 \subsetneq \stdFace_2 \subsetneq \cdots \subsetneq \stdFace_{d-(k+1)} \subsetneq \Lambda_+^{(k)}(p,e),
	\]
	which has length $d-k+1$.
	Now, if $q$ is a minimal polynomial for $\Lambda_+^{(k)}(p,e)$, the rank function with respect to $q$ must strictly increase along the chain of faces {by item~\ref{prop:rank:3} of Proposition~\ref{prop:rank}}. So $q$ must have degree at least $d-k$. 
	Since $D_{e}^kp$ has degree $d-k$, it must be minimal as well. 
	This shows 
	item~\ref{prop:strict:2}.
	
	Finally, $\Lambda_+^{(k)}(p,e) = \Lambda_+^{(k-1)}(p,e)$ cannot hold because that would imply that $\Lambda_+(D_{e}^kp,e)= \Lambda_+(D_{e}^{k-1}p ,e)$ and this would contradict the minimality of $D_{e}^kp$. 	 This shows item~\ref{prop:strict:3}.
	\end{proof}

We conclude this subsection with a discussion on how the faces of ROG hyperbolicity cones are also ROG hyperbolicity cones themselves. 
First, let us recall that a face of any hyperbolicity cone is also a hyperbolicity cone.
This is discussed in detail in \cite[Section~3.1]{LRS21}. The basic idea is that if $\stdFace \face \Lambda_+$, $z \in \reInt F$  and $z$ has multiplicity $m$, then differentiating $m$-times the polynomial $p$  along $e$ leads to a hyperbolic polynomial which, when restricted to $\spanVec \stdFace$,
is hyperbolic along $z$ and generates the face $\stdFace$. More precisely, we have the following.

\begin{proposition}[{\cite[Corollary~3.4]{LRS21}}]\label{prop:faces}
	Suppose that $p$ is hyperbolic with respect to $e$ and $\stdFace$ is a face of $\Lambda_+(p,e)$. 
	Let $z\in \reInt(\stdFace)$,  $m=\textup{mult}(z)$ 
	and denote the restriction of  $D_e^{m}p$ to the subspace $\spanVec{\stdFace}$ by $q$.
	Then,  $q$ is hyperbolic with respect to $z$ and $\stdFace = \Lambda_+(q,z)$ holds.
\end{proposition}
With that, we are ready to prove the following.
\begin{proposition}[Faces of ROG hyperbolicity cones are also ROG]\label{prop:faces_are_rog}
Under the setting of Proposition~\ref{prop:faces}, if $\Lambda_+$ is regular and ROG with respect to $p$ and $e$, then $\stdFace$ is ROG with respect to $q$ and $z$.
\end{proposition}
\begin{proof}
Suppose that $p$ has degree $d$, so the rank of $z$ is $r \coloneqq d-m$ and the degree of $q$ is also $r$.
Let $\hat \stdFace$ be any extreme ray of $\stdFace$.
 By Proposition~\ref{prop:chain2}, there is a chain of faces of $\Lambda_+$ satisfying
\[
\{0\} = \stdFace_0 \subsetneq \stdFace_1 = \hat \stdFace \subsetneq \stdFace_2 \subsetneq \cdots \subsetneq \stdFace_{r-1} \subsetneq \stdFace_{r} = \stdFace.
\]
In particular, all the $\stdFace_{i}$ are also faces of $\stdFace$. Now, we consider the rank of $\stdFace_i$ computed with respect to $q$ and $z$ instead of computing with respect to $p$ and $e$. 
$\stdFace = \stdFace_{r}$ has rank $r$ and the rank function is strictly decreasing when we go down the chain of faces {(see item~\ref{prop:rank:3} of Proposition~\ref{prop:rank})}. 
Since we have $r+1$ faces, it must be the case that the rank of $\stdFace_{i}$ is $i$, when computed with respect to $q$ and $z$. In particular, $\hat \stdFace$ has rank~$1$.
\end{proof}

\subsubsection{Examples of ROG hyperbolicity cones}\label{sec:rog_ex}
In this subsection, we present two families of ROG hyperbolicity cones.

\paragraph{Symmetric cones}

Let $\jAlg$ be a real  finite-dimensional Euclidean space.  A cone $\stdCone \subseteq \jAlg$
is said to be \emph{homogeneous} if for every $x, y \in \reInt \stdCone$, there exists $A \in \Aut(\stdCone)$ such that $Ax = y$.
Then, a cone is said to be \emph{self-dual} if there exists an inner-product $\inProd{\cdot}{\cdot}$ under which 
the dual cone $\stdCone^* \coloneqq \{y \in \jAlg \mid \inProd{x}{y} \geq 0, \forall x \in \stdCone\}$ coincides with $\stdCone$.
Finally, a cone is said to be \emph{symmetric} if it is both homogeneous and self-dual.
Typical examples of symmetric cones include the nonnegative orthant, the second-order cones, the cone of real symmetric positive semidefinite matrices and their direct products.

The symmetric cones are exactly the ones that arise as cone of squares of Euclidean Jordan Algebras. More precisely, a Euclidean Jordan Algebra is a finite-dimensional Euclidean space equipped with an inner-product $\inProd{\cdot}{\cdot}$ and a bilinear product $\jProd{}{}:\jAlg \times \jAlg \to \jAlg$ satisfying the following properties.
\begin{enumerate}[$(i)$]
	\item $\jProd{x}{y} = \jProd{y}{x}$,
	\item $\jProd{x}({\jProd{x^2}{y}}) = \jProd{x^2}({\jProd{x}{y}})$, where $x^2 = \jProd{x}{x}$,
	\item $\inProd{\jProd{x}{y}}{z} = \inProd{x}{\jProd{y}{z}}$,
\end{enumerate}
for every $x,y,z \in \jAlg$. Every Euclidean Jordan algebra has an identity element $e \in \jAlg$, so that 
$\jProd{e}{x} = x$ holds, for all $x \in \jAlg$.  Then, the corresponding cone of squares is
defined by
\[
\stdCone \coloneqq \{\jProd{x}{x} \mid x \in \jAlg \}
\]
and $\stdCone$ is a symmetric cone \cite[Theorem~III.2.1]{FK94} such that $e$ belongs to the interior of $\stdCone$. Conversely, every symmetric cone arises in this way, see~\cite[Theorem~III.3.1]{FK94}.
For more information on Jordan Algebras see \cite{FK94,K99,FB08}.

G\"uler showed in \cite[Theorem~8.1]{Gu97} that all homogenous cones are hyperbolicity cones so, in particular, all symmetric cones are hyperbolicity cones as well. See also \cite[Section~2.2]{SS08} for a related discussion. 

Here, we will observe that all symmetric cones are not only hyperbolicity cones but they can also be realized as ROG hyperbolicity cones. 
For those familiar with Jordan Algebras this might be almost obvious, but we explain here the relevant details briefly.

Let $\jAlg$ be a Euclidean Jordan algebra with Jordan product $\jProd{}{}$ and inner-product $\inProd{\cdot}{\cdot}$. An element $c \in \jAlg$ is said to be  an \emph{idempotent} if $c^2 = c$. An idempotent $c$ is said to be \emph{primitive} if it cannot be written as the sum of two non-trivial idempotents $a,b$ satisfying $\jProd{a}{b} = 0$.  

Every element $x \in \jAlg$ admits a \emph{spectral decomposition} as follows. 
There are	 primitive idempotents $c_1, \dots, c_r$ satisfying
$c_1 + \cdots + c_r  = \stdInt$, $\jProd{c_i}{c_j} = 0$ for $i \neq j$ and  
unique real numbers $\lambda _1, \ldots, \lambda_r$ satisfying
\begin{equation}		
x = \sum _{i=1}^r \lambda _i c_i \label{eq:dec},
\end{equation}
see \cite[Theorem~III.1.2]{FK94}.
The number $r$ (which depends only on the algebra $\jAlg$ and not on the specific $x$) is called the \emph{rank} of the algebra $\jAlg$. 
In analogy with classical linear algebra, the $\lambda_i$ are sometimes called the \emph{eigenvalues of $x$} and, although they may be repeated, they are uniquely defined for $x$. 
From the spectral theorem and properties of the Jordan product, one may prove the following relations.
\begin{equation}\label{eq:eig_jordan}
x \in \stdCone \Leftrightarrow \lambda_{i} \geq 0, \forall i  \qquad x \in \reInt\stdCone \Leftrightarrow \lambda_{i} > 0, \forall i.
\end{equation}
Next, let $\lambda_i(x)$ denote the $i$-th smallest eigenvalue of $x$ and let $\lambda(x) \coloneqq (\lambda_{1}(x),\cdots, \lambda_r(x))$. We define the function $\det_{\jAlg}:\jAlg \to \Re$ given by
\begin{equation}\label{eq:det_j}
\textstyle\det_{\jAlg}(x) \coloneqq \lambda_{1}(x)\cdots\lambda_{r}(x).
\end{equation}
Let us check that $\det_{\jAlg}$ is a hyperbolic polynomial along $e$. First,  $\det_{\jAlg}$ is indeed a polynomial because of \cite[Theorem~III.1.2]{FK94}, which implies that the elementary symmetric polynomials composed with the eigenvalue map $\lambda$ are polynomials.
Next, we observe that for  $t\in \Re$, the eigenvalues of the identity element $e$ are $1$, so $\det_{\jAlg}(e) = 1$. Then,  the (Jordan Algebraic) eigenvalues of 
$x - te$ are exactly $\lambda_1(x) - t, \cdots, \lambda_r(x) - t$ and they are all real. 
Therefore, for every $x \in \jAlg$, the roots of 
$t \mapsto \det_{\jAlg}(x - te)$ are the eigenvalues of $x$. This shows that $\det_{\jAlg}$ is a hyperbolic along $e$ and, in view of \eqref{eq:eig_jordan}, the underlying hyperbolicity cone is indeed $\stdCone$.

Finally, let $\stdFace \face \stdCone$ be an extreme ray generated by $x \in \stdCone$. By the spectral theorem, $x$ can be written as a nonnegative linear combination of primitive idempotents as in \eqref{eq:dec}. Since $c_i^2 = c_i$, each $c_i \in \stdCone$. Therefore, if $\lambda_{i} > 0$ for some $i$, then since $\stdFace$ is a face, $x = \alpha c_i$ holds for some $\alpha > 0$. 
We have $\jProd{c_i}{c_j} = 0$ for $i \neq j$, which implies that $c_i$ and $c_j$ cannot be parallel. 
This implies that at most one eigenvalue of $x$ can be positive. On the other hand, at least one must be positive (since we would have $x = 0$ otherwise). 
Therefore, \emph{exactly} one eigenvalue of $x$ is positive and the rank of $x$ with respect to $\det_{\jAlg}$ is $1$.
This shows that $\stdCone$ is a ROG hyperbolicity cone with respect to $\det_{\jAlg}$. We note this as a proposition.
\begin{proposition}[Symmetric cones are hyperbolicity ROG cones]\label{prop:sym_rog}
	Let $\stdCone \subseteq \jAlg$ be a finite-dimensional symmetric cone and $\jAlg$ its underlying Euclidean Jordan Algebra. Then $\stdCone$ is a ROG hyperbolicity cone with respect to 
$\det_{\jAlg}$ in \eqref{eq:det_j}.
\end{proposition}

We speculate that the same might be true for homogeneous cone but this seems to be a more complicated question.
G\"uler showed that homogeneous cones are hyperbolicity cones as well in \cite[Section~8]{Gu97}.
However, it is not clear whether the polynomial $p$ given therein is minimal,  which is a necessary condition for the cone to be ROG with respect to $p$, in view of Proposition~\ref{prop:strict}.



\paragraph{ROG spectrahedral cones}
A closed convex cone $\stdCone \subseteq \Re^m$ is said to be \emph{spectrahedral} \cite{RG95} if  for some nonnegative $n$ there are  $m$ matrices $A_i \in \S^n$ such that \[
\stdCone = \{ x \in \Re^m \mid x_1 A_1 + \cdots + x_m A_m \in \PSDcone{n} \}.
\]
Defining the linear map $\stdMap(x) \coloneqq x_1 A_1 + \cdots + x_m A_m $, $\stdCone$ can be alternatively written as 
\begin{equation}\label{eq:spec}
\stdCone = \{ x \mid \stdMap(x) \in \PSDcone{n} \}.
\end{equation}
Without loss of generality, we may assume that the $A_i$ are linearly independent, so that $\stdMap$ is a linear bijection between 
$\stdCone$ and the cone of ``slack matrices'' $\stdCone_S \coloneqq \matRange(\stdMap) \cap \PSDcone{n}$, which corresponds to a linear slice of the cone of positive semidefinite matrices.
Put otherwise, a cone is spectrahedral if it is linearly isomorphic to an intersection of the form $\PSDcone{n} \cap \stdSpace$, where $\stdSpace \subseteq \S^n$ is a subspace.

A spectrahedral cone as in \eqref{eq:spec} is said to be \emph{non-degenerate} if there exists $\bar{x} \in \stdCone$ such that $\stdMap(\bar{x})$ is positive definite. If 
$\stdCone$ is non-degenerate, then $ \det_{\stdMap}: \Re^m \to \Re$ given by 
{$\det_{\stdMap}(x) \coloneqq \det(\stdMap(x))$ is a hyperbolic polynomial along $\bar{x}$ such that $\stdCone = \Lambda_+(\det_{\stdMap} , \bar{x})$.}

Finally, $\stdCone$ as in \eqref{eq:spec} is said to be \emph{rank-one generated} (ROG) \cite{Hd16} if the following condition holds: if $x$ generates an extreme ray of $\stdCone$ then the rank of the matrix $\stdMap(x)$ is $1$.
As mentioned in the introduction, one reason why the study of ROG spectrahedral cones is important is because whether a spectrahedral cone is ROG or not is intimately connected to whether  SDP relaxations of certain quadratic problems are exact or not, see \cite[Lemma~1.2]{Hd16} and \cite{AKW22}.

\begin{example}[ROGness depends on $\stdMap$]
	A spectrahedral cone $\stdCone$  may satisfy \eqref{eq:spec} for different choices of $\stdMap$. 
	It is well-known that the 3D second-order cone $\{(x_0,x_1,y_2) \mid x_0 \geq 0, x_0^2 \geq x_1^2 + x_2^2\}$ can be represented as a spectrahedral cone over the $2\times 2$ matrices or over the $3\times 3$ matrices (e.g., see \cite[Example~3.5]{Hd16}). In particular,
	following \cite[Example~3.5]{Hd16}, let
	\begin{equation}
	\stdMap_1(x) \coloneqq \begin{pmatrix}
	x_0 & x_1 & x_2\\
	x_1 & x_0 & 0 \\
	x_2 & 0  & x_0
	\end{pmatrix}, \qquad \stdMap_2(x) \coloneqq \begin{pmatrix}
	x_0 +x_1 & x_2\\
	x_2 & x_0-x_1
	\end{pmatrix}.
	\end{equation}
	With that, $\stdCone$ is ROG with respect to the representation induced by $\stdMap_2$ but not with respect to the representation induced by $\stdMap_1$.
\end{example}

If $\stdCone$ is non-degenerate and ROG as a spectrahedral cone (with respect to the representation given in \eqref{eq:spec}), it is not immediately obvious that $\stdCone$, seen as a hyperbolicity cone, is also ROG with respect to $\det_{\stdMap}$.
In the next proposition, we take care of this issue. 

\begin{proposition}\label{prop:rog_rank}
	Let $\stdCone$  (as in \eqref{eq:spec}) be a non-degenerate spectrahedral cone. 
	Then, for every $x \in \stdCone$ we have 
	that $\rank(x)$ (computed with respect to $\det_{\stdMap}$) is equal to the rank of the matrix $\stdMap(x)$.
	In particular, $\stdCone$ is a ROG spectrahedral cone (with respect to the representation induced by $\stdMap$) if and only if it is a ROG hyperbolicity cone with respect to $\det_{\stdMap}$.
\end{proposition}
\begin{proof}
	Since $\stdCone$ is non-degenerate, there exists $\bar{x}$ such that $\stdMap(\bar{x})$ is positive definite.
	By definition, 	$\rank(x)$ (computed with respect to $\det_{\stdMap}$) is equal to number of nonzero roots that 
	the one-dimensional polynomial 
	\[
	t \mapsto \det(\stdMap(x) - t\stdMap(\bar{x}))
	\]
	has. Now,  $\det: \S^n \to \Re$ is a hyperbolic polynomial along the identity matrix and the corresponding hyperbolicity cone is $\PSDcone{n}$. Therefore, $\det$ is also hyperbolic along $\stdMap(\bar{x})$ and, since the rank function does not depend on the hyperbolicity direction (\cite[Proposition~22]{Re06}), the number of nonzero roots of  $t \mapsto \det(\stdMap(x) - t\stdMap(\bar{x}))$ coincide with the number of nonzero roots of 
	$t \mapsto \det(\stdMap(x) - tI_n)$, where $I_n$ is the $n\times n$ identity matrix. However, this latter is equal to the usual matrix rank of $\stdMap(x)$.
\end{proof}
{{
		Proposition~\ref{prop:rog_rank} needs to be interpreted carefully. In fact, there are spectrahedral cones  that are hyperbolic ROG but none of their spectrahedral representations are ROG. So, in fact, the class of cones that are hyperbolic ROG under some choice of $p$ is strictly larger than the class of cones that are spectrahedral ROG. This will be illustrated right after we take care of the following technical lemma.
		
		\begin{lemma}\label{lem:non_deg}
			If $\stdCone$ is a spectrahedral cone as in \eqref{eq:spec}, then $\stdCone$ also has a non-degenerate spectrahedral representation $\stdCone = \{x \in \Re^m \mid \hat \stdMap(x) \in \PSDcone{r}\}$ for some $r \leq n$. Furthemore, $\hat \stdMap$ can be chosen in such a way that $\stdCone$ is a ROG spectrahedral cone with respect to $\stdMap$ if and only if $\stdCone$ is a ROG spectrahedral cone with respect to $\hat \stdMap$.
		\end{lemma}
		\begin{proof}
		This is a consequence of the facial structure of $\PSDcone{n}$ and it is discussed to some extent in \cite[Lemma~2.5]{Hd16}, so we will only present a sketch  of the proof here. Suppose that $\stdCone$ is a spectrahedral cone as in \eqref{eq:spec}, let $\stdSpace \subseteq \S^n$ the space spanned by the $A_i$ and let $\stdFace$  be the minimal face of $\PSDcone{n}$ containing $\stdSpace$. With that,  $\stdSpace \cap (\reInt \stdFace) \neq \emptyset$. Since $\stdFace$ is a face of $\PSDcone{n}$ it is linearly isomorphic to a positive semidefinite cone over matrices of size $r \leq n$. In fact, more can be said and there exists an invertible matrix $V$ such that $V \stdFace V^\T = \left\{\begin{pmatrix}
		A & 0 \\ 0 & 0 
		\end{pmatrix} \mid A \in \PSDcone{r} \right\} $.  Letting $\pi_r : \S^n \to \S^r$ be the map that takes a $n\times n$ symmetric matrix to its upper left $r \times r $ block, we construct the map $\hat{\stdMap}: \Re^{m} \to \S^r$ given by $\hat{\stdMap}(x) = \pi_{r}(V \stdMap(x)V^\T)$. With that, 
		$\stdCone = \{x \mid\hat{\stdMap}(x) \in \PSDcone{r}  \}$  and this a non-degenerate spectrahedral representation of $\stdCone$. By construction, $\stdCone$ is ROG with respect to $\stdMap$ if and only if it is ROG with respect to $\hat \stdMap$.				
		\end{proof}
		
		\begin{proposition}\label{prop:socp_not_rog}
			For $n \geq 3$, the $(n+1)$-dimensional second-order cone 
			$\SOC{2}{n+1} \coloneqq \{x \in  \Re^{n+1} \mid x_0 \geq 0, x_0^2 \geq x_1^2 + \cdots + x_n^2 \}$
			is ROG as a hyperbolicity cone but none of its spectrahedral representations are ROG.
		\end{proposition}
		\begin{proof}
Second-order cones are symmetric cones, so they are ROG hyperbolicity cones by Proposition~\ref{prop:sym_rog}.
This can also be shown directly by the well-known fact that $\SOC{2}{n+1}  = \Lambda_+(p,e)$ where  $p(x) \coloneqq x_0^2 - x_1^2 - \cdots - x_n^2$ and $e \coloneqq (1,0,\cdots,0)$.	

Next, suppose that we have a ROG spectrahedral representation of 
$\SOC{2}{n+1}$ so that 
\[
\SOC{2}{n+1} =  \{ x \in \Re^{n+1}\mid \stdMap(x) \in \PSDcone{m} \},
\]
where $m \geq 2$, $\stdMap(x) \coloneqq x_0 A_0 + \cdots x_{n} A_n$ with $A_i \in \S^m$. In view of Lemma~\ref{lem:non_deg}, we may assume that the representation is non-degenerate without loss of generality.
Then, by Proposition~\ref{prop:rog_rank}, $\SOC{2}{n+1}$ must be a ROG hyperbolicity cone with respect to $\det_{\stdMap}$. By Proposition~\ref{prop:strict}, $\det_{\stdMap}$ is a  minimal degree polynomial for $\SOC{2}{n+1}$ and, therefore, must have  degree $2$. In particular, the maximum rank that a matrix $\stdMap(x)$ can have for $x \in \SOC{2}{n+1}$ is also $2$. Since we assumed that the spectrahedral representation is non-degenerate, we conclude that $m = 2$.
However, the dimensions of $\PSDcone{2}$, $\SOC{2}{n+1}$ is $3$ and $n + 1$, respectively.

Finally, since $\SOC{2}{n+1}$ is pointed, $\stdMap$ must be injective, so 
the dimension of $\SOC{2}{n+1}$ and $
\stdMap(\SOC{2}{n+1})$ coincide. As $\stdMap(\SOC{2}{n+1})$ is contained in $\PSDcone{2}$, it must be the case that $n+1 \leq 3$.
In conclusion, $\SOC{2}{n+1}$ cannot have a ROG spectrahedral representation in dimension $4$ or higher.
		\end{proof}
	}

}

\subsection{Automorphisms of ROG hyperbolicity cones}
In this subsection, we prove our main results concerning the automorphisms of ROG hyperbolicity cones and their derivative relaxations.

\begin{theorem}[Automorphism groups do not enlarge along derivative relaxations]\label{theo:hyper_aut}
Let $\Lambda_+=\Lambda_+(p,e)$ be a hyperbolicity cone with $d:=\deg p\geq 4$ and $\dim \Lambda_+ \geq 3$. Suppose that $\Lambda_+$ is regular and ROG with respect to $p$. 
Then, for $k$ with $1\leq k\leq d-3$ we have\[
\Aut(\Lambda_+^{(k)}) \subseteq \Aut(\Lambda_+).
\] 
\end{theorem}
\begin{proof}
Let $\mathcal{F}_2$ and $\mathcal{F}_2^{(k)}$  be the set of all strictly convex faces with dimension at least $2$ of $\Lambda_+$ and  $\Lambda_+^{(k)}$, respectively.
An initial observation is that $\mathcal{F}_2$ and $\mathcal{F}_2^{(k)}$ only contain boundary faces of $\Lambda_+$ and $\Lambda_+^{(k)}$, respectively.
Put otherwise, neither $\Lambda_+$ nor 
$\Lambda_+^{(k)}$ are strictly convex under the assumptions on $d$ and $k$.
This follows, for example, from the proof of Proposition~\ref{prop:strict}, where we showed that $\Lambda_+$ and $\Lambda_+^{(k)}$ contain chain of faces of length $d+1$ and $d-k+1$ ($\geq 4$), respectively. Whereas the longest chain of faces in a strictly convex cone is no longer than $3$.

 Next, we will prove that
\[
\mathcal{F}_2 = \mathcal{F}_2^{(k)}.
\]
For the ``$\subseteq$'' inclusion, let $F \in \mathcal{F}_2$ and let $x \in \reInt \stdFace$. Then, $x$ has rank $2$ by item~\ref{lem:rank:3} of Lemma~\ref{lem:rank}, i.e., it has multiplicity $d-2$. By Lemma~\ref{lem:face-mul}, $F$ is a face of 
$\Lambda_+^{(d-3)}$.
Furthermore, since the dimension of $F$ is at least two, $F$ is not an extreme ray.
Recalling that $\Lambda_+^{(d-3)}= (\Lambda_+^{(k)})^{(d-k-3)}$, applying Lemma~\ref{lem:face} to $\Lambda_+^{(k)}$ and
$(\Lambda_+^{(k)})^{(d-k-3)}$ we conclude that the face $F$ (which is not an extreme ray) must be a face of $\Lambda_+^{(k)}$.
Conversely, let $F \in \mathcal{F}_2^{(k)}$.
Since $F$ has dimension at least two, 
Lemma~\ref{lem:face} implies that $F$ is a face of $\Lambda_+$. Since it is strictly convex, $F \in \mathcal{F}_2$.

Therefore,  $\mathcal{F}_2 = \mathcal{F}_2^{(k)}$ holds.
Now, let $A \in \Aut(\Lambda_+^{(k)})$.
Since $A$ is a bijective linear map it preserves strict convexity and the dimension of cones. Therefore $
A\mathcal{F}_2^{(k)} = \mathcal{F}_2^{(k)}$ which leads to 
\begin{equation}\label{eq:af2}
A\mathcal{F}_2 = \mathcal{F}_2,
\end{equation}
that is, $A$ permutes the set of rank $2$ faces of $\Lambda_+$, which is $\mathcal{F}_2$ by items~\ref{lem:rank:1} and \ref{lem:rank:3} of Lemma~\ref{lem:rank}.

Let $E$ be an arbitrary extreme ray of $\Lambda_+$ which must, by assumption, have rank $1$ so it is generated by some $x\in \Lambda_+$ of rank $1$. Since the dimension of $\Lambda_+$ is at least $3$, there is another extreme ray generated by some $y \in \Lambda_+$ distinct from $E$. Now, the rank of 
$\Lambda_{+}$ is $d \geq 4$ and 
the face $F(x+y)$ has rank $2$ (item~\ref{lem:rank:2} of Lemma~\ref{lem:rank}), so we have $F(x+y) \neq \Lambda_{+}$.
In particular, we can find yet another extreme ray generated by some $z \in \Lambda_+$ such that $z \not \in F(x+y)$.
From  Lemma~\ref{lem:min_face}, we have  
$x,y \in F(x+y)=F(x,y)$ and $x,z \in F(x+z)=F(x,z)$. Since $x$, $y$ and $z$ all generate distinct extreme rays, the dimensions of $F(x,y)$ and $F(x,z)$ are both at least two.
Then, since $x+y$ and $x+z$ both have rank $2$ (item~\ref{lem:rank:2} of Lemma~\ref{lem:rank}), 
$F(x,y), F(x,z) \in \mathcal{F}_2$.

An intersection of faces is a face, so $F(x,y) \cap F(x,z)$ contains $E$ and is a face of both $\Lambda_+$ and $F(x,z)$.  Since $F(x,y) \cap F(x,z)$ is contained in $F(x,z)$ (which has rank $2$), in view of Lemma~\ref{lem:rank} and the strict convexity of $F(x,z)$,
$F(x,y) \cap F(x,z)$ is either $F(x,z)$ or $E$. Since $z \not \in F(x,y)$, we conclude that 
$F(x,y) \cap F(x,z) = E$.

Then, from \eqref{eq:af2} we have $AF(x,y), AF(x,z) \in \mathcal{F}_2$ and, since $A$ is an bijection, we have:
\[
A(E) = A(F(x,y)\cap F(x,z)) = (AF(x,y)) \cap (AF(x,z)).
\]
 
 In particular, $A(E) = (AF(x,y)) \cap (AF(x,z))$ is a face of $\Lambda_+$ (since it is an intersection of faces) with dimension $1$ (since $E$ has dimension $1$), so it is an extreme ray of $\Lambda_+$. We conclude that 
 $A$ maps an extreme ray of $\Lambda_+$ to another extreme ray of $\Lambda_+$. 
 Everything we have done so far also applies to $A^{-1}$, so we conclude that $A$ permutes the set of extreme rays of $\Lambda_+$.
 Since a pointed closed convex cone is the convex hull of its extreme rays, $A$ must be, in fact, an automorphism of $\Lambda_+$.	
\end{proof}

Next, we will  strengthen Theorem~\ref{theo:hyper_aut}.
Our result relies on the following G{\aa}rding's inequality.

\begin{lemma}[{G{\aa}rding \cite[Theorem 5]{Gar}}]
\label{lem:Garding}
Let $p:\Re^n\to \Re$ be a hyperbolic polynomial along $e$ and let $d=\deg p$.
Let $P(x_1,\ldots,x_d)$ be the polar form of $p$:
\begin{equation}\label{eq:polar}
P(x_1,\ldots,x_d) = \frac{1}{d!}\frac{\partial}{\partial t_1} \cdots \frac{\partial}{\partial t_d} p(t_1 x_1 + \cdots + t_d x_d) \Big|_{t_1=\cdots =t_d=0} = \frac{1}{d!}\nabla^d p(0)[x_1,\ldots,x_d].
\end{equation}
Then, for any $x_1,\ldots,x_d \in \reInt\Lambda_{+}(p,e)$, we have
$$
 p(x_1)^{\frac{1}{d}} \cdots p(x_d)^{\frac{1}{d}} \leq  P(x_1,\ldots,x_d).
$$
The equality holds if and only if $x_1,\ldots,x_d$ are pairwise proportional modulo $\Lambda_+\cap (-\Lambda_+)$.
\end{lemma}

We can now prove the following result.
\begin{theorem}
\label{theo:hyper_aut_2}
Let $\Lambda_+=\Lambda_+(p,e)$ be a hyperbolicity cone with $d:=\deg p\geq 4$ and $\dim \Lambda_+ \geq 3$. Suppose that $\Lambda_+$ is regular and ROG with respect to $p$.
Then, we have for all $1\leq k \leq d-3$ that
\[
\Aut(\Lambda_+^{(k)}) = \{A \in \Aut(\Lambda_+) \mid A(\Re_+e)=\Re_+e \},
\]
where $\Lambda_+^{(k)} = \Lambda_+(D_e^kp,e)$ is the $k$-th derivative relaxation of $\Lambda_+$ with respect to $e$.
\end{theorem}
\begin{proof}
We recall that, by Proposition~\ref{prop:strict}, $p$ and $D^k_ep$ are minimal degree polynomials for $\Lambda_+$ and $\Lambda_+^{(k)}$, respectively.

We first prove the  ``$\supseteq$''. Suppose that $A \in \Aut(\Lambda_+)$ and $Ae = \alpha e$ for some $\alpha>0$. Then, Proposition~\ref{prop:aut} implies that
there exists $\kappa > 0$ such that
\[
p = \kappa \cdot p \circ A
\]
Observe that 
\[
D_e^k p(x) = \frac{d^k p(x+te)}{dt^k} \Big|_{t=0} = \nabla^k p(x)[e^k],
\]
where $e^k$ stands for the tuple $(e,\ldots,e)$ of length $k$.
By taking the $k$-th derivative on the both sides of $p = \kappa \cdot p\circ A$ and using $Ae=\alpha e$, we have
\[
D_e^kp(x) \equiv \nabla^k p(x)[e^k] = \kappa \nabla^k p(Ax)[{(Ae)^k}] = \kappa \alpha^k \nabla^k p(Ax)[e^k] = \kappa \alpha^k ((D_e^k p)\circ A)(x).
\]
Combining with the fact $Ae = \alpha e \in \reInt \Lambda_+^{(k)}$ and the minimality of $D_e^kp$, we invoke Proposition~\ref{prop:aut} in order to conclude that $A \in \Aut(\Lambda_+^{(k)})$.

Next, we show the ``$\subseteq$'' inclusion. Let $A \in \Aut(\Lambda_+^{(k)})$.
By Theorem~\ref{theo:hyper_aut}, we also have $A \in \Aut(\Lambda_+)$. So it remains to prove that $Ae = \alpha e$ for some $\alpha>0$.

Proposition~\ref{prop:aut} implies that there exist $\kappa_1,\kappa_2>0$ satisfying
$$
p = \kappa_1 \cdot p \circ A,\qquad \kappa_2 \cdot D_e^k p =  (D_e^k p) \circ A.
$$
Computing the $k$-th derivative of $p=\kappa_1 \cdot p\circ A$ yields $D_e^k p(x) = \kappa_1  D_e^k (p\circ A)$. This leads to
\[
(D_e^k p)\circ A = \kappa_2 \cdot D_e^k p = \kappa_1 \kappa_2 D_e^k (p \circ A),
\]
in other words,
\[
\nabla^k p(Ax)[e^k] = \kappa_1\kappa_2 \nabla^k p(Ax)[(Ae)^k],\quad \forall x \in \Re^n.
\]
As $A$ is nonsingular, we obtain
\begin{equation}\label{eq:hyper_aut_2_proof}
\nabla^k p(z)[(Ae)^k] = \kappa \nabla^k p(z)[e^k],\quad \forall z \in \Re^n,
\end{equation}
where $\kappa:=(\kappa_1 \kappa_2)^{-1}>0$.

Let $P(x_1,\ldots,x_d)=\nabla^d p(0)[x_1,\ldots,x_d]/d!$ be the polar form of $p$ as defined in \eqref{eq:polar}.
Now setting $z=te$  in \eqref{eq:hyper_aut_2_proof} and calculating $d^{d-k} / dt^{d-k}$ at $t=0$ on the both sides, one has
\begin{align}
P((Ae)^k,e^{d-k}) &= \nabla^d p(0)[(Ae)^k,e^{d-k}]/d! = \kappa \nabla^d p(0)[e^d]/d! = \kappa \cdot P(e,\ldots,e) \notag\\
&=\kappa  \cdot p(e)^{\frac{d}{d}} \quad (\text{by Lemma~\ref{lem:Garding}}) \notag\\
&= \kappa \kappa_1^{\frac{k}{d}}  \cdot p(e)^{\frac{d-k}{d}}  p(Ae)^{\frac{k}{d}} \label{eq:hyper_aut_2_proof_2}\\
&\leq \kappa \kappa_1^{\frac{k}{d}} P((Ae)^k,e^{d-k}) \quad (\text{by Lemma~\ref{lem:Garding}}). \notag
\end{align}
Dividing  both sides by $P((Ae)^k, e^{d-k})$ ($=\kappa \cdot p(e)>0$), we see that
$$
\kappa \kappa_1^{\frac{k}{d}} \geq 1.
$$
Similarly, setting $z = tAe$ in \eqref{eq:hyper_aut_2_proof} and calculating $d^{d-k} / dt^{d-k}$ at $t=0$ on the both sides, one has
\begin{align*}
\kappa P(e^k,(Ae)^{d-k}) &= \kappa \nabla^d p(0)[e^k,(Ae)^{d-k}]/d! = \nabla^d p(0)[(Ae)^d]/d! = P(Ae,\ldots,Ae) \\
&=p(Ae)^{\frac{d}{d}} \quad (\text{by Lemma~\ref{lem:Garding}}) \\
&= \kappa_1^{-\frac{k}{d}}p(e)^{\frac{k}{d}} p(Ae)^{\frac{d-k}{d}}\\
&\leq \kappa_1^{-\frac{k}{d}} P(e^k,(Ae)^{d-k}) \quad (\text{by Lemma~\ref{lem:Garding}}).
\end{align*}
Since $\kappa P(e^k,(Ae)^{d-k})=p(Ae)=p(e)/\kappa_1$ is positive,
this yields $\kappa\kappa_1^{\frac{k}{d}} \leq 1$ and so $\kappa\kappa_1^{\frac{k}{d}} = 1$ holds. Then,  \eqref{eq:hyper_aut_2_proof_2} becomes
\[
P((Ae)^k,e^{d-k}) = p(e)^{\frac{d-k}{d}}  p(Ae)^{\frac{k}{d}}.
\]
By Lemma~\ref{lem:Garding} and the pointedness of $\Lambda_+$, this occurs only when $Ae = \alpha e$ for some $\alpha > 0$.
\end{proof}

The restrictions on $d$ and $k$ in  Theorem~\ref{theo:hyper_aut_2} only leave out the derivative relaxations such that $D_e^kp$ has degree less or equal than $2$. 
Also, the ROG assumption is essential for our results. In the next example, we see that our results do not hold for the $\ell_1$-cone.

\begin{example}[$\ell_1$-cone]\label{ex:l1}
Consider the $\ell_1$-cone in $\Re^3$
$$
\{x \in \Re^3 ~|~ x_3 \geq |x_1|+|x_2| \} = \{x ~|~ x_3 + (-1)^{i}x_1 + (-1)^j x_2  \geq 0,~\forall i,j\in \{0,1\}\},
$$
which is a hyperbolicity cone $\Lambda_+ = \Lambda_+(p,e)$ with
$$p(x)=(x_3 + x_1 + x_2) (x_3 + x_1 - x_2) (x_3 - x_1 + x_2) (x_3 - x_1 - x_2),\quad e = (0,0,1)^T. $$
{The polynomial $p$ is of minimal degree defining $\Lambda_+$.
Indeed, note that any minimal degree polynomial, say $q$, must be a divisor of $p$ in the polynomial ring $\Re[x]$ by Proposition~\ref{prop:hv}. Since polynomials of degree one are irreducible and $\Re[x]$ is, in particular, an unique factorization domain, $q$ must be a product of some of the four linear polynomials $p_{ij}(x)=x_3 + (-1)^{i}x_1 + (-1)^j x_2$ and a constant. In particular, $\Lambda_+(q,e)$ is the intersection of the corresponding half-spaces $p_{ij}(x)\geq 0$, which coincides with $\Lambda_+(p,e)$ only when $p=\kappa q$ for some constant $\kappa$.
}

We remark that $\Lambda_+$ is not ROG since every extreme ray of $\Lambda_+$ has rank two.
The derivatives of $p$ are given by
\[
D_e p(x)=4 x_3(x_3^2-x_1^2-x_2^2),\quad D_e^2p(x)=4 (3 x_3^2-x_1^2 - x_2^2).
\]
The derivative relaxation $\Lambda_+^{(1)}$ is the second-order cone in $\Re^3$  and the factor $x_3$ in $D_ep$ is redundant, in particular $D_e p(x)$ is not of minimal degree and shows that Proposition~\ref{prop:strict} may fail if the cone is not ROG.
Since $\Lambda_+^{(1)}$  is a symmetric cone, for any $\hat e \in \reInt \Lambda_+^{(1)}$, there exists $A \in \Aut(\Lambda_+^{(1)})$ such that $Ae = \hat e$. 
In addition, $\reInt \Lambda_+ \subseteq \reInt(\Lambda_+^{(1)})$, which is a consequence of the interlacing properties between the eigenvalues with respect to $p$ and $D_ep$, e.g., \cite[Section~4]{Re06}.
Since  $\Lambda_+ \neq \Lambda_+^{(1)}$, there exists $\hat e \in \reInt \Lambda_+^{(1)} $ such that $ \hat e \not \in \Lambda_+$. 

Letting $A \in \Aut(\Lambda_+^{(1)})$ be such that $Ae = \hat e$, $A$ does not fix $e$ nor belong to $\Aut(\Lambda_+)$, which shows that both Theorems~\ref{theo:hyper_aut} and \ref{theo:hyper_aut_2} fail for the $\ell_1$-cone.

\end{example}

\paragraph{Generalized Perron-Frobenius Theorem and a converse of Theorem~\ref{theo:hyper_aut_2}}
Theorem~\ref{theo:hyper_aut_2} tells us that, subject to a condition on $k$,  automorphisms of the $k$-th derivative relaxation of a regular ROG hyperbolicity cone are the automorphisms of the original cone that have the hyperbolic direction $e$ as an eigenvector. 
We will close this section with a converse of sorts. We will show that, reciprocally, every automorphism of a ROG hyperbolicity cone must already be the automorphism of some derivative relaxation. The  caveat is that the derivative relaxation in which the automorphism will be found may be a relaxation of a \emph{face} of the cone.

In order to do that, we need a discussion on the generalized Perron-Frobenius theorem.
The classical Perron-Frobenius theorem implies that a nonzero $n\times n$ nonnegative matrix has a nonnegative eigenvector associated to a positive eigenvalue. This can be summarized by 
saying that the condition  
$A\Re^n_+ \subseteq \Re^n_+$ (i.e., $A$ is nonnegative) implies that $A$ has an eigenvector that belongs to $\Re^n_+$ as well.
This result has several generalizations where 
$\Re^n_+$ is replaced with an arbitrary closed convex cone, see \cite{Van68,BS75} and 
\cite[Chapter~1]{BP94}.
In particular, the following holds.

\begin{theorem}[{\cite[Theorem~3.1]{Van68}}]\label{theo:pf}
Let $\stdCone \subseteq \Re^n$ be a regular closed convex cone and suppose that $A$ is a $n\times n$ real matrix satisfying $A \stdCone \subseteq \stdCone$. Then, $\stdCone$ contains an eigenvector of $A$ corresponding to the spectral radius of $A$ (i.e., the maximum of the absolute values of the eigenvalues of $A$). 
\end{theorem}
An immediate consequence of Theorem~\ref{theo:pf} is that if $A$ is an automorphism of $\stdCone$, then $A$ has an eigenvector contained in $\stdCone$. In fact, a bit more can be said about this.

\begin{proposition}\label{prop:min_face_eig}
Let $\stdCone \subseteq \Re^n$ be a regular closed convex cone and let $A \in \Aut(\stdCone)$. If $z \in \stdCone$ is an eigenvector of $A$, then $A(\stdFace(z)) = \stdFace(z)$.
Put otherwise, $A \in \Aut(\stdFace(z))$ holds.
\end{proposition}
\begin{proof}
$A$ is an automorphism, so $A\stdFace(z)$ must be a face of $\stdCone$ as well. Since $A$ is a linear map, $\reInt (A (\stdFace(z))) = A (\reInt \stdFace(z))$ holds, e.g. \cite[Theorem~6.6]{RT97}. 
Then, since $z \in \reInt \stdFace(z)$ (see \eqref{eq:fs}) and $z$ is an eigenvector of $A$ (which is associated to a positive eigenvalue since $\stdCone$ is pointed), we obtain that 
$z \in \reInt (A \stdFace(z))$. 
We conclude that the relative interiors of the faces $\stdFace(z)$ and $A (\stdFace(z))$ intersect, so they must coincide by \eqref{eq:faces_eq}, i.e., $A \stdFace(z) = \stdFace(z)$.
\end{proof}

Gathering everything, we obtain the following converse of Theorem~\ref{theo:hyper_aut_2}.

\begin{theorem}\label{theo:aut_converse}
Let $\Lambda_+  = \Lambda_+(p,e) \subseteq \Re^n$ be a regular ROG hyperbolicity cone and let $A \in \Aut(\Lambda_+)$. Let $z \in \Lambda_+$ be an eigenvector of $A$ (at least one exists by Theorem~\ref{theo:pf}) and let $\stdFace(z) \face \Lambda_+$ be the minimal face of $\Lambda_+$ that contains $z$. The following statements hold.
\begin{enumerate}[$(i)$]
	\item\label{theo:aut_converse:1} $A \in \Aut(\stdFace(z))$.
	\item\label{theo:aut_converse:2} Let $m$ be the multiplicity of $z$ with respect to $p$ and let $q$ be as in Proposition~\ref{prop:faces}, so that $\stdFace(z) = \Lambda_+(q,z)$. If $d-m \geq 4$ then 	for $1 \leq k \leq d-m-3$, we have
	\[
	A \in \Aut(\Lambda_+^{(k)}(q,z)),
	\]  
	where $\Lambda_+^{(k)}(q,z)$ is the $k$-th derivative relaxation of $\Lambda_+(q,z)$ along the direction $z$.
\end{enumerate}
\end{theorem}
\begin{proof}
Item~\ref{theo:aut_converse:1} follows from Proposition~\ref{prop:min_face_eig}, so we move on to item~\ref{theo:aut_converse:2}.

In view of Propositions~\ref{prop:faces}  and \ref{prop:faces_are_rog}, $\stdFace$ is ROG with respect to $q$ and $z$. Since $\Lambda_+$ is pointed, $\stdFace$ is also pointed, so, with respect to $\spanVec \stdFace$, $\stdFace$ is a regular ROG hyperbolicity cone.
Therefore, we can apply Theorem~\ref{theo:hyper_aut_2} to $\Lambda_+(q,z)$, which leads to 
\[
\Aut(\Lambda_+^{(k)}(q,z)) = \{\hat A \in \Aut(\Lambda_+(q,z)) \mid \hat A(\Re_+z)=\Re_+z \},
\]
for $k $ satisfying $1 \leq k \leq d - m -3 $. In particular, since $z$ is an eigenvector of $A$ and $\Lambda_+$ is pointed, we have $A(\Re_+ z) = \Re_+z$, so 
$A \in \Aut(\Lambda_+^{(k)}(q,z)) $ holds.
\end{proof}
Theorem~\ref{theo:hyper_aut_2} and Theorem~\ref{theo:aut_converse} taken together can be summarized as follows. 
The automorphisms of the derivative relaxations of order $1 \leq k \leq d -3$ of a regular ROG hyperbolicity cone $\Lambda_+(p,e)$ are exactly the automorphisms of $\Lambda_+(p,e)$ which have $e$ as an eigenvector. 
Conversely, in view of Proposition~\ref{prop:min_face_eig}, every automorphism $A$ of $\Lambda_+(p,e)$ must also be an automorphism of at least one non-zero face $\stdFace$ of $\Lambda_+(p,e)$ containing an eigenvector $z$ of $A$ in its relative interior. Such a face is also a hyperbolicity cone (Proposition~\ref{prop:faces}) and $A$ must also be an automorphism of the derivative relaxations of order $1 \leq k \leq d-m-3$ of $\stdFace$ along  $z$.

The case where $m = 0$ in Theorem~\ref{theo:aut_converse} is noteworthy. In this case, the automorphism $A$ has an eigenvector $z$ in the interior of $\Lambda_+(p,e)$ and the derivative relaxations that appear in Theorem~\ref{theo:aut_converse} are relaxations of the original cone $\Lambda_+(p,e)$ along the interior direction $z$.

Furthermore, if $A$ has a \emph{single} eigenvector $z$ in $\Lambda_+$ (up to scalar multiples) and $z$ is in the interior of $\Lambda_+$, then \cite[Theorem~4.2]{Van68} or \cite[Chapter~1, Theorem~3.16]{BP94} implies that $A$ is what is called a \emph{$\Lambda_+$-irreducible matrix}, which means that for a face $\stdFace\face \Lambda_+$ then $A\stdFace \subseteq \stdFace$ never holds except for  $\stdFace = \{0\}$ or $\stdFace = \Lambda_+$. 
The notion of $\stdCone$-irreducible matrix for an arbitrary cone $\stdCone$ generalizes the concept of \emph{irreducible matrix} that appear in the classical Perron-Frobenius theory.
In our particular case, since $A$ is an automorphism, $\Lambda_+$-irreducibility means that $A$ permutes the set of faces of $\Lambda_+$ but fixes no face except the trivial ones $\{0\}$ and $\Lambda_+$.

\section{Applications}\label{sec:app}
In this section, we collect a few applications of the results so far.
\subsection{Automorphisms of $\Re_+^{n,(k)}$ and $\S_+^{n,(k)}$ }
In this subsection, we take a closer look at the derivative relaxations of 
$\Re^n_+$ and $\PSDcone{n}$. 
These cones and other closely related objects form an interesting test-bed for ideas and conjectures about hyperbolic polynomials and have been studied by many authors \cite{Zin08,Sa13,Br14,SP15,Sa18,Kum21}, with quite a few works devoted to questions related to their spectrahedral representability in connection with the generalized Lax conjecture. 

The nonnegative orthant $\Re_+^n$ can be realized as a hyperbolicity cone as follows
$$\Re_+^n = \Lambda_+(p,e),\quad p(x)=x_1x_2\cdots x_n,\quad e = (1,1,\ldots,1)^T.$$
Then, the following fact is well-known\footnote{This follows from the fact that an automorphism of $\Re^n_+$ must permute the $n$ extreme rays of $\Re^n_+$ and those extreme rays are generated by the usual coordinate basis $e_1,\ldots, e_n$. 
	Alternatively, since $\Re^n_+$ is the direct product of $n$ copies of $\Re_+$, the result can be derived from general results on automorphisms of direct sums of cones, see \cite[Section~4 and Lemma~4.1]{Horne78}.}:
$$
\Aut(\Re_+^n) = \{\Diag(c_1,\ldots,c_n) P \mid c_1,\ldots,c_n > 0,~ P \text{ is a permutation matrix} \}.
$$
The derivative $D_e^kp(x)$ for $k\geq 0$ is a positive multiple of the elementary symmetric polynomial of degree $n-k$, that is,
$$
D_e^kp(x) = k! \, s_{n-k}(x), \text{ where } s_d(x) = \sum_{1\leq i_1<\cdots<i_d \leq n} x_{i_1}\cdots x_{i_d}.
$$
By \eqref{eq:hyp_deriv}, the $k$-th derivative relaxation $\Re_+^{n,(k)}\coloneqq\Lambda_+^{(k)}$ has the description
\begin{equation}\label{eq:nonneg-orth-deriv}
\Re_+^{n,(k)} = \{x \in \Re^n \mid s_i(x) \geq 0,~i=1,\ldots,n-k\},\quad k=1,\ldots,n-1.
\end{equation}
We remark that $\Re_+^{n,(n-1)}$ for $n\geq 2$ is the half-space $\{x \in \Re^n \mid x_1 + \cdots + x_n \geq 0\}$.

{
Moreover, $\Re_+^{n,(n-2)}$ is linearly isomorphic to the Lorentz cone $\SOC{2}{n}$ as $D_e^{n-2}p$ is quadratic and the corresponding matrix has one positive eigenvalue and $n-1$ negative eigenvalues (see \cite[pg.~486]{BGLS01} for a related discussion). In particular, the automorphism group of $\SOC{2}{n}$ has the following characterization \cite{LS75}: For $A \in \GL_n(\Re)$, we have
\begin{equation}\label{eq:aut-SOC}
 A \SOC{2}{n} = \SOC{2}{n} \text{ or } A \SOC{2}{n} = -\SOC{2}{n}
 ~~ \iff ~~
 A^T J A = \mu J \text{ for some } \mu>0,
\end{equation}
where $J = \Diag(-1,1,\ldots,1)$.
Now let $\hat{J} \in \S^n$ be the symmetric matrix satisfying $s_2(x)=x^T\hat{J}x$.
For $\Re_+^{n,(n-2)} = \{x\,|\,s_2(x)\geq 0,s_1(x)\geq 0\} = \{x\,|\,x^T\hat{J}x\geq 0,~e^Tx\geq 0\}$, one has the following characterization: for $A \in \GL_n(\Re)$,
\begin{equation}\label{eq:aut-nonnegorth-2}
A\Re_+^{n,(n-2)} = \Re_+^{n,(n-2)} \text{ or } A\Re_+^{n,(n-2)}=-\Re_+^{n,(n-2)}
~~\iff~~
 A^T \hat{J} A = \mu \hat{J} \text{ for some } \mu>0.
\end{equation}
In fact, by Sylvester's law of inertia, there exists $B \in \GL_n(\Re)$ such that $\hat{J} = B^TJB$, which yields the relation
$$
\SOC{2}{n}\cup-\SOC{2}{n} = \{x\,|\,x^TJx\geq 0\} = B\{x\,|\,x^T\hat{J}x\geq 0\} = B(\Re_+^{n,(n-2)}\cup -\Re_+^{n,(n-2)}).
$$
The characterization \eqref{eq:aut-nonnegorth-2} follows by this relation combined with \eqref{eq:aut-SOC}.
}

The next result determines the structure of $\Aut(\Re_+^{n,(k)})$ for $k=1,\ldots,n-3$.

\begin{theorem}\label{th:aut-orth}
For $n \geq 4$ and $k=1,\ldots,n-3$, we have
$$
\Aut(\Re_+^{n,(k)}) = \{\alpha P \mid \alpha >0,~ P \text{ is a permutation matrix}\}.
$$
\end{theorem}
\begin{proof}

The extreme rays of $\Re_+^n$ consist of the rays generated by the coordinate basis $e_1,\ldots,e_n$ and the hyperbolicity cone $\Re_+^n$ is ROG with respect to $p(x)=x_1x_2\cdots x_n$ and $e=(1,\ldots,1)^T$. By Theorem~\ref{theo:hyper_aut_2}, for $k = 1,\ldots,n-3$ we have
\[
\Aut(\Re_+^{n,(k)}) = \{A \in \Aut(\Re_+^n) \mid A(\Re_+ e) = \Re_+ e \}.
\]
If $A = \alpha P$ for $\alpha > 0$ and $P$ a permutation matrix, since $A(e) = \alpha Pe =\alpha e$, we have $A \in \Aut(\Re_+^{n,(k)})$.

Conversely, let  $A \in \Aut(\Re_+^{n,(k)})$.
Then, $A \in \Aut(\Re_+^n)$ and so
$$A = \Diag(c_1,\ldots,c_n) P $$
for some $c_1,\ldots,c_n > 0$ and a permutation matrix $P$. 
Since $A(\Re_+e) = \Re_+ e$, there exists $\alpha > 0$ such that $Ae = \alpha e$ and thus
$$
(\alpha,\ldots,\alpha) = \alpha e = Ae = \Diag(c_1,\ldots,c_n) P e = \Diag(c_1,\ldots,c_n) e =  (c_{1},\ldots,c_{n}).
$$
This implies that $c_1 = \cdots = c_n = \alpha$, i.e, $A = \alpha P$.
\end{proof}
Next, we analyze the automorphisms of the derivative relaxations of the cone of positive semidefinite matrices $\PSDcone{n}$, which is a ROG hyperbolicity cone realized as
\[
\S_+^n = \Lambda_+(P,I_n),\quad P(X) = \det X,
\]
where $I_n$ is the $n\times n$ identity matrix. 

First, we need a discussion on linear operators of $\S^n$.
Here we remark that if 
$L: \S^n \to \S^n$ is a linear operator, $L$ might be completely oblivious to the underlying matrix structure. In particular, even if $L$ is a bijection, it is not necessarily the case that $L$ maps nonsingular matrices to nonsingular matrices\footnote{Consider for example, the bijective linear map  $L$ that takes $\begin{psmallmatrix}
	a & b \\ b & c
	\end{psmallmatrix}$ to $\begin{psmallmatrix}
	b & a \\ a & c
	\end{psmallmatrix}$. The rank $2$ matrix $\begin{psmallmatrix}
	0 & 1 \\ 1 & 0
	\end{psmallmatrix}$ gets mapped to the rank $1$ matrix 
	$\begin{psmallmatrix}
	1 & 0 \\ 0 & 0
	\end{psmallmatrix}$.}.

In contrast, a family of linear operators that \emph{do} preserve the rank is given by the operators $L_M:\S^n\to \S^n$ defined by 
\begin{equation}\label{eq:lm}
L_M(X) \coloneqq MXM^\T, \qquad \forall X \in \S^n
\end{equation}
where $M$ is a fixed nonsingular matrix $M \in \GL_n(\Re)$ and $M^\T$ is its adjoint. 
In what follows we say that 
$L$ is \emph{rank $1$ preserver} if $L(X)$ has rank $1$ whenever $X \in \S^n$ has rank $1$.
Then, a remarkable (but perhaps not widely known) result tells us that the operators of the form $\pm L_M$ are the only linear bijections that are rank $1$ preservers.
\begin{theorem}[{Lim \cite[Theorem~10]{Li79} and Waterhouse \cite[Theorem~11]{Wa89}}]\label{theo:lw}
	Let $L:\S^n \to \S^n$ be a bijective linear operator that is a rank $1$ preserver. Then, there exists a nonsingular matrix $M\in \GL_n(\Re)$ and nonzero scalar $\alpha$ 
	such that $L = \alpha L_M$.
\end{theorem}

We now have all the tools necessary to prove the following result.

\begin{theorem}\label{theo:aut_snk}
For $n\geq 4$ and $k$ with $1\leq k\leq n-3$, we have
$$
\Aut(\S_+^{n,(k)}) = \{\alpha L_Q \mid \alpha > 0,~Q \text{ is a $n\times n$ orthogonal matrix}\}.
$$
\end{theorem}
\begin{proof}
The PSD cone $\PSDcone{n}$ is a ROG hyperbolicity cone generated by the determinant polynomial along the direction $I_n$, where $I_n$ is the $n\times n$ identity matrix. By Theorem~\ref{theo:hyper_aut_2}, we have
\[
\Aut(\S_+^{n,(k)}) = \{L \in \Aut(\PSDcone{n}) \mid L(\Re_+I_n)=\Re_+I_n \}.
\]
A map of the form $\alpha L_Q$ with $\alpha > 0$, $Q$ orthogonal is an automorphism of $\PSDcone{n}$ that satisfies the condition $\alpha L_Q(\Re_+I_n) = \Re_+I_n$, so this shows the ``$\supseteq$'' inclusion.

Next, we prove the converse. Let $L$ be an automorphism of $\Aut(\S_+^{n,(k)})$. In particular, it must be an automorphism of $\PSDcone{n}$.
The extreme rays of $\PSDcone{n}$ are generated by rank~$1$ matrices, so an automorphism of $\PSDcone{n}$ must 
be a rank~$1$ preserver. 
Theorem~\ref{theo:lw} then implies that  $L = \alpha L_M$ for some nonzero scalar $\alpha$ and matrix $M \in \GL_n(\Re)$ as in \eqref{eq:lm}.
Since $\PSDcone{n}$ is pointed, $\alpha$ must be positive and, rescaling $M$ if necessary, we may assume that $\alpha = 1$.
However, the condition
$L(\Re_+I_n)=\Re_+I_n$ implies that there exists $\kappa > 0$ such that 
\[
MM^\T = \kappa I_n,
\]
that is, $Q \coloneqq M/\sqrt{\kappa}$ is an orthogonal matrix and $\kappa L_Q = L_M = L$. 
%
\end{proof}
The derivative relaxations of $\S_+^{n,(k)}$ and $\Re_+^{n,(k)}$ are intimately connected as follows (cf. \eqref{eq:nonneg-orth-deriv} and \cite[Eq.\,(3)]{Re06}).
For a symmetric matrix $X \in \S^n$, let $\lambda(X)$ denote the eigenvalue map $\lambda(X)=(\lambda_1(X),\ldots,\lambda_n(X))$. Then,
\begin{equation}\label{eq:snk_rnk}
\S_+^{n,(k)} = \lambda^{-1}(\Re_+^{n,(k)}) = \{X \in \S^n \mid \lambda(X) \in \Re_+^{n,(k)}\},\quad k \geq 0.
\end{equation}
That is, $\S_+^{n,(k)}$ is the \emph{spectral cone} generated by  $\Re_+^{n,(k)}$. 
In view of Theorems~\ref{th:aut-orth} and \ref{theo:aut_snk}, a natural question is how to relate the automorphism groups of a spectral cone and its underlying permutation invariant subset.
We conclude this subsection with a detour on this topic.

\paragraph{Automorphisms of spectral cones}
We say that a set $C \subseteq \Re^n$ is \emph{permutation invariant} if for every permutation matrix $P$ we have $PC = C$.
The spectral set associated to $C$ is given by 
\[
\lambda^{-1}(C)  \coloneqq \{X \in \S^n \mid \lambda(X) \in C \}.
\]
Spectral sets and functions over symmetric matrices and, more generally, over Euclidean Jordan algebras have been studied in quite detail in several works, e.g., see \cite{Le96,B07,SS08,JG16,JG17} and many others. Still, to the best of our knowledge, the following result 
connecting $\Aut(\stdCone)$ and $\Aut(\lambda^{-1}(\stdCone))$
 seems to be novel.

\begin{theorem}
	Let $\stdCone$ be a permutation invariant closed convex cone and  let \break$L \in \Aut(\lambda^{-1}(\stdCone))$. Then, the following statements hold.
	\begin{enumerate}[$(i)$]\label{thm:spec_aut}
		\item\label{thm:spec_aut:1} 
		Suppose that $L = \alpha L_M$ for some $M \in \GL_n(\Re)$ and $\alpha \in \{-1,1\}$. If $D$ is a diagonal matrix corresponding to the singular values of $M$ in any order, then $\alpha D^2 \in \Aut(\stdCone)$.
		
		\item\label{thm:spec_aut:2} 
		If $L$ is a rank $1$ preserver, then 
		$L = \alpha L_M$ for some $M \in \GL_n(\Re)$ and $\alpha \in \{-1,1\}$.
		If $\stdCone$ is nonzero and pointed then $\alpha  = 1$.

	\end{enumerate}
\end{theorem}
\begin{proof}
	We start with item~\ref{thm:spec_aut:1}.
	We write the singular value decomposition of $M$ as $M = UDV$, so that $D$ is a diagonal matrix containing the singular values of $M$ and $U,V$ are orthogonal matrices.
	Since $\lambda^{-1}(\stdCone)$ is a spectral cone and $U,V$ are orthogonal, $L_U$ and $L_V$ both belong to $\Aut(\lambda^{-1}(\stdCone))$.
	We have
	\begin{align*}
	U\lambda^{-1}(\stdCone) U^\T = \lambda^{-1}(\stdCone) = \alpha UDV \lambda^{-1}(\stdCone) V^\T DU^\T = 
	\alpha UD \lambda^{-1}(\stdCone)  DU^\T.
	\end{align*}
	Therefore, $\lambda^{-1}(\stdCone) = \alpha D\lambda^{-1}(\stdCone) D$.
	That is, $\alpha L_D \in \Aut(\lambda^{-1}(\stdCone))$.
	
	Let $x \in \stdCone$, then $\text{Diag}(x) \in \lambda^{-1}(\stdCone)$. Therefore \[ \alpha D\text{Diag}(x)D = \alpha D^2 \text{Diag}(x) =   \text{Diag}(\alpha D^2x) \in \lambda^{-1}(\stdCone).\]
	This implies that $\alpha D^2 x \in \stdCone$, that is, $\alpha D^2 \stdCone \subseteq \stdCone$. Now, recall that $\Aut(\lambda^{-1}(\stdCone))$ is a group so  $L^{-1} = \alpha L_{M^{-1}}$ is an automorphism of $\lambda^{-1}(\stdCone)$ as well.
	Observing that all we have done so far also applies to $L^{-1}$, we conclude 
	that $\alpha D^{-2} \stdCone \subseteq \stdCone$ and that 
	$\alpha D^2 \in \Aut(\stdCone)$. This concludes the proof of item~\ref{thm:spec_aut:1}.

	Since $\alpha L_M = \alpha/\abs{\alpha} L _{M\sqrt{\abs{\alpha}} }$,
	the first half of item~\ref{thm:spec_aut:2} is a direct consequence of Theorem~\ref{theo:lw}, so that $\alpha$ and $M$ can be normalized in a way that $\alpha \in \{-1,1\}$.
	Next,  suppose that $\stdCone$ is nonzero and pointed.
	Then, \cite[Lemma~7.2]{JG16} tells us that either $e = (1,1,\ldots,1)^T$ or $-e$ belongs to $\stdCone$, but not both.
	So let $f \in \stdCone$ be $e$ in the former case and $-e$ in the latter.
	 Since $\alpha D^2 \in \Aut(\stdCone)$ (by item~\ref{thm:spec_aut:1}), we have  
that	 $\alpha D^2f$ belongs to $\stdCone$ as well. 
Let $\text{Sym}(n)$ denote the group of $n\times n$ permutation matrices. Because $\stdCone$ is convex
	we have that
	\[
	z \coloneqq \frac{1}{|\text{Sym}(n)|}\sum _{P \in \text{Sym}(n)} P(\alpha D^2 f)
	\]
	belongs to $\stdCone$ as well.
	Since $Pz = z$ holds for every $P \in \text{Sym}(n)$,
	all components of $z$ are the same and equal to a positive multiple of the sum of components of $\alpha D^2 f$.
	For the sake of obtaining a contradiction, suppose that $\alpha = -1$. Then, the  components of $\alpha D^2 f$ have the \emph{opposite sign} of the components of $f$  and $z$ is of the form $\beta f$, where $\beta < 0$. Since $\stdCone$ is a cone, this implies that both $f$ and $-f$ belong to $\stdCone$ which contradicts the fact that $\stdCone$ is pointed. 
	We conclude that  $\alpha = 1$. 
\end{proof}
\begin{remark}[Automorphisms of $\S_+^{n,(k)}$ and $\Re_+^{n,(k)}$ revisited]
Theorem~\ref{thm:spec_aut} and Theorem~\ref{theo:hyper_aut} are enough to prove Theorem~\ref{theo:aut_snk} without making use of Theorem~\ref{theo:hyper_aut_2}.	
A sketch of this is as follows. For $n \geq 4$ and $k$ such that $1 \leq k \leq n-3$, Theorem~\ref{theo:hyper_aut} implies that an automorphism $A$ of $\S_+^{n,(k)}$ is an automorphism of $\PSDcone{n}$. 
Then, since the extreme rays of $\PSDcone{n}$ correspond to rank $1$ matrices, $A$ must be a rank~$1$ preserver, so, by Theorem~\ref{theo:lw}, it is of the format $L_M$ for some $M \in \GL_n(\Re)$. 
From Theorem~\ref{thm:spec_aut} and \eqref{eq:snk_rnk}, the diagonal matrix $D$ containing the singular values of $M$ is such that $D^2 \in \Aut(\Re_+^{n,(k)})$. 
However, by Theorem~\ref{th:aut-orth},  the only diagonal matrix in $\Aut(\Re_+^{n,(k)})$ are the multiples of the identity matrix. So, the singular values of $M$ are all equal, which implies that $M$ is of the format $\alpha L_Q$ for $\alpha > 0$ and $Q$ some orthogonal matrix. Conversely, suppose that $L$ is a linear map of the form $\alpha L_Q$. In view of \eqref{eq:snk_rnk}, membership on $\S_+^{n,(k)}$ only depends on the eigenvalues of the matrix, so $\alpha L_Q$ indeed belongs to $\Aut(\S_+^{n,(k)})$.
\end{remark}

\begin{remark}[Related results]
	Let $C\subseteq \Re^n$ be a closed cone	(not necessarily convex) and let $\stdCone_C \coloneqq \{\sum uu^T \mid u \in C \} \subseteq \S^n$, where $\sum uu^T$ denotes a finite sum of matrices of the form $uu^T$. Then, $\stdCone_C$ is called the \emph{completely positive cone generated by $C$} and it is always convex. For example if $C = \Re^n$, then $\stdCone_C = \PSDcone{n}$ and if $C = \Re^n_+$, then $\stdCone_C$ is the usual cone of completely positive matrices. The completely positive cone construction is yet another way of obtaining a matrix cone from a vector cone. Gowda, Sznajder and Tao showed in \cite{GST13} how to relate the automorphism groups of $C$, $C \cup -C$ and $\stdCone_C$ under appropriate assumptions.
	We note that, similarly, the results of Lim and Waterhouse played an important role in the proof of \cite[Theorem~1]{GST13}.
\end{remark}

\subsection{Non-homogeneity and the Lyapunov rank}\label{sec:non_hom}
As mentioned in Section~\ref{sec:int}, at least from the optimization point of view, hyperbolicity cones can be seen as a natural step after  symmetric and homogeneous cones.
Here, we recall that a closed convex cone $\stdCone$ is said to be \emph{homogeneous} if $\Aut(\stdCone)$ acts transitively on the relative interior of $\stdCone$, i.e., for any $x,y \in \reInt\stdCone$, there exists $A \in \Aut(\stdCone)$
such that $Ax=y$. 
With that, we can prove the following.
\begin{corollary}\label{cor:nonhomog}
	Let $\Lambda_+(p,e)$ be a a regular hyperbolicity cone that is ROG with respect to $p$.
	Then $\Lambda_+^{(k)}(p,e)$ is not homogeneous for $1 \leq k \leq d- 3$.
\end{corollary}
\begin{proof}
 Theorem~\ref{theo:hyper_aut_2}  implies that 
every automorphism of $\Lambda_+^{(k)}$ for $k=1,\ldots,d-3$ must fix the direction $e$, so there is no room for the automorphisms of $\Lambda_+^{(k)}(p,e)$ to act transitively.
\end{proof}

Corollary~\ref{cor:nonhomog} implies 
that $\Re_+^{n,(k)}$ and $\S_+^{n,(k)}$ are not homogeneous for $k=1,\ldots,n-3$.
And, in general, the informal conclusion is that derivative relaxations can be significantly poorer in automorphisms when compared with the original cone.

Recalling Section~\ref{sec:int}, the Lyapunov rank $\beta(\stdCone)$ of a cone $\stdCone$ is the dimension of the Lie algebra of the automorphism group of $\stdCone$ \cite{GT14_2,GT14,Or22}. The Lie algebra corresponds to the tangent space at the identity element, so the Lyapunov rank is simply the dimension of $\Aut(\stdCone)$ as a smooth manifold.
With this in mind, the next corollary gives a quantitative statement regarding how small are
the automorphism groups of  $\Re_+^{n,(k)}$ and $\S_+^{n,(k)}$, for $k=1,\ldots,n-3$.

\begin{corollary}[Lyapunov rank of derivative relaxations of $\Re^n_+$ and $\PSDcone{n}$]\label{col:lya}
For $n \geq 4$ and  $k=1,\ldots,n-3$, 
we have
\[
\beta(\Re_+^{n,(k)}) = 1, \qquad \beta(\S_+^{n,(k)}) = \frac{n^2-n+2}{2}. 
\]
\end{corollary}
\begin{proof}
From Theorem~\ref{th:aut-orth}, the automorphisms of
$\Re_+^{n,(k)}$ are of the form $\gamma P$, where $\gamma > 0$ and $P$ is a permutation matrix. By continuity, a differentiable curve $\alpha:(-\epsilon, \epsilon) \to \Aut(\Re_+^{n,(k)})$ passing through the identity element at $\alpha(0)$ must be confined to the one dimensional ray $\{\gamma I \mid \gamma > 0 \}$ for sufficiently small $\epsilon$, where $I$ is the $n\times n$ identity matrix. This implies that the tangent space at the identity has dimension $1$, so $\beta(\Re_+^{n,(k)}) = 1$.
 
Next, we turn our attention to $\S_+^{n,(k)}$. Consider the Lie group $\mathcal{G} \coloneqq \Re_{++} \times O(n)$, which is the direct product of $\Re_{++}$, the multiplicative group of positive reals and $O(n)$, the group of $n\times n$ real orthogonal matrices.
Then, we define the following Lie group homomorphism $\psi: \mathcal{G} \mapsto \Aut(\S_+^{n,(k)})$ given by 
\[
\psi(\gamma, Q) \coloneqq \gamma L_Q,
\]
where $\gamma \in \Re_{++}$, $Q \in O(n)$ and $L_Q$ is as in \eqref{eq:lm}.
Denote the kernel of $\psi$ by $\ker \psi$.
 By Theorem~\ref{theo:aut_snk}, $\psi$ is surjective. 
Therefore, the quotient group $\mathcal{G}/\ker \psi$ is diffeomorphic to $\Aut(\S_+^{n,(k)})$, e.g., see \cite[Theorem~21.27]{Lee2012}. However, the dimension of $\mathcal{G}/\ker \psi$ is given by 
\[
\dim(\mathcal{G}) - \dim(\ker\psi),
\]
e.g., see \cite[Theorem~21.17]{Lee2012}.
Furthermore, $\ker \psi = \{(1,I),(1,-I)\}$ holds\footnote{This can be seen by noticing that if $(\lambda, Q)  \in \ker \psi$, then $\lambda QIQ^\T = I$ holds so that $\lambda = 1$. 
Then, letting $e_i \in \Re^n$ denote the $i$-th coordinate vector we have $Qe_ie_i^TQ^\T = e_i e_i^T$, which implies that the $i$-th column of $Q$ must be $\pm e_i$. Finally, if the nonzero elements in two columns $i,j$ of $Q$ have different sign, then $Q(e_ie_j^T + e_je_i^T)Q^\T = -(e_ie_j^T+e_je_i^T) \neq e_ie_j^T + e_je_i^T$. We conclude that $Q = \pm I$.  }, which is a discrete subgroup of $\mathcal{G}$, so its manifold dimension is $0$.
We conclude that  the dimension of $\Aut(\S_+^{n,(k)})$ coincides with the dimension of  $\mathcal{G}$. The group $O(n)$ has dimension $n(n-1)/2$ (e.g., \cite[Example~7.28]{Lee2012}) and $\Re_{++}$ has dimension $1$, which leads to the formula 
$\beta(\S_+^{n,(k)}) = \frac{n^2-n+2}{2}$.
\end{proof}
The Lyapunov rank of $\Re^n_+$ and $\PSDcone{n}$ are $n$ and $n^2$, respectively, see \cite[pg.~166]{GT14_2}.  
They are also examples of \emph{perfect cones} \cite[Theorem~6]{GT14_2}. A necessary and sufficient condition for a regular cone $\stdCone$ to be perfect is that $\beta(\stdCone) \geq \dim \stdCone$ \cite[Theorem~1]{OG16}.
From Corollary~\ref{col:lya}, for $n \geq 4$ and  $k=1,\ldots,n-3$, 
 the $k$-th derivative relaxations of $\Re^n_+$ and $\PSDcone{n}$ are also not perfect. This is yet another way in which the automorphism group of derivative relaxations can be much poorer than that of the original cone.


\section{Some open questions}\label{sec:conc}
In this paper, we proved basic results for ROG hyperbolicity cones and provided a formula that relates the automorphism group of a regular ROG cone with the automorphism group of its derivative relaxations.  We then applied the results to compute the automorphisms of derivative relaxations of $\Re^n_+$ and $\PSDcone{n}$.
We conclude this work with a few open questions.
\begin{itemize}
	\item \emph{Are homogeneous cones ROG?} G\"uler showed in \cite{Gu97} that homogeneous cones are indeed hyperbolicity cones. 
	However, it is not clear whether they are also ROG. 
	\item \emph{For which class of hyperbolicity cones does the formula in Theorem~\ref{theo:hyper_aut_2} hold?} 
	It might be interesting to examine whether Theorem~\ref{theo:hyper_aut_2} (or, more modestly, Theorem~\ref{theo:hyper_aut}) can be extended beyond ROG hyperbolicity cones. Example~\ref{ex:l1} already points some of the difficulties in this task. The $\ell_1$-cone is, in a sense, the next best thing after a ROG cone, since all the extreme rays still have the same rank. Nevertheless, Theorem~\ref{theo:hyper_aut} does not hold for it. 
	One of the difficulties is that, in general, even if $p$ is a minimal polynomial for $\Lambda_+ = \Lambda_+(p,e)$, it is not necessarily the case that $D_{e}p$ is a minimal polynomial for $\Lambda_+^{(1)}$. 
	When this happens, the automorphism group of $\Lambda_+^{(1)}$ might enlarge, instead of shrink. 
\end{itemize}
With regard to the second question, we observe that the proof of the first half of Theorem~\ref{theo:hyper_aut_2} leads to the following partial result.
\begin{proposition}\label{prop:partial}
Let $\Lambda_+=\Lambda_+(p,e)$ be a hyperbolicity cone with $d:=\deg p$ and $\dim \Lambda_+ \geq 3$. Let $k$ satisfy $1 \leq k \leq d-2$ and suppose that $p$ and $D_e^kp$ are minimal polynomials for $\Lambda_+$ and $\Lambda_+^{(k)}$, respectively. Then
\[
\Aut(\Lambda_+^{(k)}) \supseteq \{A \in \Aut(\Lambda_+) \mid A(\Re_+e)=\Re_+e \}.
\]
\end{proposition}
\begin{proof}
It is the exact same proof of the ``$\supseteq$'' inclusion in Theorem~\ref{theo:hyper_aut_2}. 
This proof only depends on the minimality of $p$, $D_e^kp$ and uses Proposition~\ref{prop:aut}, which does not require that the cone be ROG nor regular.
\end{proof}
{
	\small{
		\section*{Acknowledgements}
We thank the referees for their  comments, which helped to improve the paper.
}}

\bibliographystyle{abbrvurl}
\bibliography{bib_plain}

\begin{thebibliography}{10}

\bibitem{AKW22}
C.~J. Argue, F.~K\i{}l\i{}n\c{c}-Karzan, and A.~L. Wang.
\newblock Necessary and sufficient conditions for rank-one-generated cones.
\newblock {\em Mathematics of Operations Research}, 2022.
\newblock \href {https://doi.org/10.1287/moor.2022.1254}
  {\path{doi:10.1287/moor.2022.1254}}.

\bibitem{B07}
M.~Baes.
\newblock Convexity and differentiability properties of spectral functions and
  spectral mappings on {E}uclidean {J}ordan algebras.
\newblock {\em Linear Algebra and its Applications}, 422(2):664 -- 700, 2007.

\bibitem{BS75}
G.~Barker and H.~Schneider.
\newblock Algebraic {P}erron-{F}robenius theory.
\newblock {\em Linear Algebra and its Applications}, 11(3):219--233, 1975.

\bibitem{Bar73}
G.~P. Barker.
\newblock The lattice of faces of a finite dimensional cone.
\newblock {\em Linear Algebra and its Applications}, 7(1):71--82, 1973.

\bibitem{BGLS01}
H.~H. Bauschke, O.~Güler, A.~S. Lewis, and H.~S. Sendov.
\newblock Hyperbolic polynomials and convex analysis.
\newblock {\em Canadian Journal of Mathematics}, 53(3):470–488, 2001.

\bibitem{BP94}
A.~Berman and R.~J. Plemmons.
\newblock {\em Nonnegative matrices in the mathematical sciences}.
\newblock SIAM, 1994.

\bibitem{Br14}
P.~Br{\"a}nd{\'e}n.
\newblock Hyperbolicity cones of elementary symmetric polynomials are
  spectrahedral.
\newblock {\em Optimization Letters}, 8(5):1773--1782, Jun 2014.

\bibitem{Br11}
P.~Brändén.
\newblock Obstructions to determinantal representability.
\newblock {\em Advances in Mathematics}, 226(2):1202--1212, 2011.

\bibitem{CH03}
C.~B. Chua.
\newblock Relating homogeneous cones and positive definite cones via
  \textit{T}-algebras.
\newblock {\em SIAM J. Optim.}, 14(2):500--506, 2003.

\bibitem{KS08}
E.~de~Klerk and R.~Sotirov.
\newblock Exploiting group symmetry in semidefinite programming relaxations of
  the quadratic assignment problem.
\newblock {\em Mathematical Programming}, 122(2):225, Oct 2008.

\bibitem{FK94}
J.~Faraut and A.~Kor\'{a}nyi.
\newblock {\em Analysis on Symmetric Cones}.
\newblock Oxford Mathematical Monographs. Clarendon Press, Oxford, 1994.

\bibitem{FB02}
L.~Faybusovich.
\newblock On {N}esterov's approach to semi-infinite programming.
\newblock {\em Acta Applicandae Mathematica}, 74(2):195--215, Nov 2002.

\bibitem{FB08}
L.~Faybusovich.
\newblock Several {J}ordan-algebraic aspects of optimization.
\newblock {\em Optimization}, 57(3):379--393, 2008.

\bibitem{Gar}
L.~G{\aa}rding.
\newblock An inequality for hyperbolic polynomials.
\newblock {\em Journal of Mathematics and Mechanics}, 8(6):957--965, 1959.

\bibitem{KP04}
K.~Gatermann and P.~A. Parrilo.
\newblock Symmetry groups, semidefinite programs, and sums of squares.
\newblock {\em Journal of Pure and Applied Algebra}, 192(1):95--128, 2004.

\bibitem{GST13}
M.~S. Gowda, R.~Sznajder, and J.~Tao.
\newblock The automorphism group of a completely positive cone and its {L}ie
  algebra.
\newblock {\em Linear Algebra and its Applications}, 438(10):3862--3871, 2013.

\bibitem{GT14_2}
M.~S. Gowda and J.~Tao.
\newblock On the bilinearity rank of a proper cone and {L}yapunov-like
  transformations.
\newblock {\em Mathematical Programming}, 147(1):155--170, Oct 2014.

\bibitem{GT14}
M.~S. Gowda and D.~Trott.
\newblock On the irreducibility, {L}yapunov rank, and automorphisms of special
  {B}ishop–{P}helps cones.
\newblock {\em J. Math. Anal. Appl.}, 419(1):172--184, 2014.

\bibitem{Gu97}
O.~G\"uler.
\newblock Hyperbolic polynomials and interior point methods for convex
  programming.
\newblock {\em Mathematics of Operations Research}, 22(2):350--377, 1997.

\bibitem{HV07}
J.~W. Helton and V.~Vinnikov.
\newblock Linear matrix inequality representation of sets.
\newblock {\em Communications on Pure and Applied Mathematics}, 60(5):654--674,
  2007.

\bibitem{Hi14}
R.~Hildebrand.
\newblock Analytic formulas for complete hyperbolic affine spheres.
\newblock {\em Beitr{\"a}ge zur Algebra und Geometrie / Contributions to
  Algebra and Geometry}, 55(2):497--520, Oct 2014.

\bibitem{Hd16}
R.~Hildebrand.
\newblock Spectrahedral cones generated by rank 1 matrices.
\newblock {\em Journal of Global Optimization}, 64(2):349--397, Feb 2016.

\bibitem{Horne78}
J.~Horne.
\newblock On the automorphism group of a cone.
\newblock {\em Linear Algebra and its Applications}, 21(2):111--121, 1978.

\bibitem{JG16}
J.~Jeong and M.~S. Gowda.
\newblock Spectral cones in {E}uclidean {J}ordan algebras.
\newblock {\em Linear Algebra and its Applications}, 509:286--305, 2016.

\bibitem{JG17}
J.~Jeong and M.~S. Gowda.
\newblock Spectral sets and functions on {E}uclidean {J}ordan algebras.
\newblock {\em Linear Algebra and its Applications}, 518:31--56, 2017.

\bibitem{KOMK01}
Y.~Kanno, M.~Ohsaki, K.~Murota, and N.~Katoh.
\newblock Group symmetry in interior-point methods for semidefinite program.
\newblock {\em Optimization and Engineering}, 2(3):293--320, Sep 2001.

\bibitem{K99}
M.~Koecher.
\newblock {\em The {M}innesota Notes on {J}ordan Algebras and Their
  Applications}.
\newblock Number 1710 in Lecture Notes in Mathematics. Springer, Berlin, 1999.

\bibitem{Kum21}
M.~Kummer.
\newblock Spectral linear matrix inequalities.
\newblock {\em Advances in Mathematics}, 384:107749, 2021.

\bibitem{Lee2012}
J.~Lee.
\newblock {\em Introduction to Smooth Manifolds}.
\newblock Graduate Texts in Mathematics. Springer New York, 2012.

\bibitem{Le96}
A.~S. Lewis.
\newblock Convex analysis on the {H}ermitian matrices.
\newblock {\em SIAM Journal on Optimization}, 6(1):164--177, 1996.

\bibitem{Li79}
M.~H. Lim.
\newblock Linear transformations on symmetric matrices.
\newblock {\em Linear and Multilinear Algebra}, 7(1):47--57, 1979.

\bibitem{LS75}
R.~Loewy and H.~Schneider.
\newblock Positive operators on the $n$-dimensional ice cream cone.
\newblock {\em Journal of Mathematical Analysis and Applications},
  49(2):375--392, 1975.

\bibitem{LRS21}
B.~F. Louren\c{c}o, V.~Roshchina, and J.~Saunderson.
\newblock Hyperbolicity cones are amenable.
\newblock {\em arXiv e-prints}, 2021.
\newblock \href {http://arxiv.org/abs/2102.06359} {\path{arXiv:2102.06359}}.

\bibitem{Or22}
M.~Orlitzky.
\newblock Tight bounds on {L}yapunov rank.
\newblock {\em Optimization Letters}, 16(2):723--728, Mar 2022.

\bibitem{OG16}
M.~Orlitzky and M.~S. Gowda.
\newblock An improved bound for the {L}yapunov rank of a proper cone.
\newblock {\em Optimization Letters}, 10(1):11--17, Jan 2016.

\bibitem{Pa00}
G.~Pataki.
\newblock The geometry of semidefinite programming.
\newblock In H.~Wolkowicz, R.~Saigal, and L.~Vandenberghe, editors, {\em
  Handbook of semidefinite programming: theory, algorithms, and applications}.
  Kluwer Academic Publishers, online version at
  \url{http://www.unc.edu/~pataki/papers/chapter.pdf}, 2000.

\bibitem{RG95}
M.~Ramana and A.~J. Goldman.
\newblock Some geometric results in semidefinite programming.
\newblock {\em Journal of Global Optimization}, 7(1):33--50, Jul 1995.

\bibitem{Re06}
J.~Renegar.
\newblock Hyperbolic programs, and their derivative relaxations.
\newblock {\em Foundations of Computational Mathematics}, 6(1):59--79, 2006.

\bibitem{RT97}
R.~T. Rockafellar.
\newblock {\em {Convex Analysis}}.
\newblock Princeton University Press, 1997.

\bibitem{RNPA11}
G.~Rudolf, N.~Noyan, D.~Papp, and F.~Alizadeh.
\newblock Bilinear optimality constraints for the cone of positive polynomials.
\newblock {\em Mathematical Programming}, 129(1):5--31, Sep 2011.

\bibitem{Sa13}
R.~Sanyal.
\newblock On the derivative cones of polyhedral cones.
\newblock {\em Advances in Geometry}, 13(2):315--321, 2013.

\bibitem{Sa18}
J.~Saunderson.
\newblock A spectrahedral representation of the first derivative relaxation of
  the positive semidefinite cone.
\newblock {\em Optimization Letters}, 12(7):1475--1486, Oct 2018.

\bibitem{Sa20}
J.~Saunderson.
\newblock Limitations on the expressive power of convex cones without long
  chains of faces.
\newblock {\em SIAM Journal on Optimization}, 30(1):1033--1047, 2020.

\bibitem{SP15}
J.~Saunderson and P.~A. Parrilo.
\newblock Polynomial-sized semidefinite representations of derivative
  relaxations of spectrahedral cones.
\newblock {\em Mathematical Programming}, 153(2):309--331, Nov 2015.

\bibitem{SS08}
D.~Sun and J.~Sun.
\newblock L\"owner's operator and spectral functions in {E}uclidean {J}ordan
  algebras.
\newblock {\em Mathematics of Operations Research}, 33(2):421--445, 2008.

\bibitem{Va09}
F.~Vallentin.
\newblock Symmetry in semidefinite programs.
\newblock {\em Linear Algebra and its Applications}, 430(1):360--369, 2009.

\bibitem{Van68}
J.~S. Vandergraft.
\newblock Spectral properties of matrices which have invariant cones.
\newblock {\em SIAM Journal on Applied Mathematics}, 16(6):1208--1222, 1968.

\bibitem{V63}
E.~B. Vinberg.
\newblock The theory of homogeneous convex cones.
\newblock {\em Trans. Moscow Math. Soc.}, 12:340--403, 1963.
\newblock (English Translation).

\bibitem{Wa89}
W.~C. Waterhouse.
\newblock Linear transformations preserving symmetric rank one matrices.
\newblock {\em Journal of Algebra}, 125(2):502--518, 1989.

\bibitem{Zin08}
Y.~Zinchenko.
\newblock On hyperbolicity cones associated with elementary symmetric
  polynomials.
\newblock {\em Optimization Letters}, 2(3):389--402, Jun 2008.

\end{thebibliography}
\end{document}